\documentclass{birkjour}
\usepackage{amssymb}
 \newtheorem{thm}{Theorem}[section]
 \newtheorem{cor}[thm]{Corollary}
 \newtheorem{lem}[thm]{Lemma}
 \newtheorem{prop}[thm]{Proposition}
 \theoremstyle{definition}
 
 \theoremstyle{remark}
 \newtheorem{rem}[thm]{Remark}
 
 \numberwithin{equation}{section}

\def\eps{\varepsilon}

\def\S{{S}}
\newcommand{\cA}{\mathcal A}
\newcommand{\cB}{\mathcal B}

\newcommand{\cD}{\mathcal D}
\newcommand{\cE}{\mathcal E}

\newcommand{\cN}{\mathcal N}

\newcommand{\cMH}{\mathcal{MH}}
\newcommand{\cSM}{\mathcal{SM}}
\newcommand{\cSX}{\mathcal{SX}}

\newcommand{\R}{\mathbb R}

\newcommand{\LL}{\mathbb L}
\newcommand{\EE}{\mathbb E}
\newcommand{\FF}{\mathbb F}

\newcommand{\nn}{\nonumber}

\newcommand{\lj}{[\![}
\newcommand{\rj}{]\!]}
\newcommand{\pd}{\partial}

\begin{document}

%
%
%
%
%
%
%
%
%

\title[Incompressible Two-Phase Flows with Phase Transitions]
 {On Incompressible Two-Phase Flows with Phase Transitions\\ and Variable Surface Tension}

\author[J.~Pr\"uss]{Jan Pr\"uss}
\address
{Institut f\"ur Mathematik \\
         Martin-Luther-Universit\"at Halle-Witten\-berg\\
         D-06099 Halle, Germany}
\email{jan.pruess@mathematik.uni-halle.de}

\author[S.~Shimizu]{Senjo Shimizu}
\address
{Department of Mathematics, Shizuoka University\\
      422-8529 Shizuoka, Japan}
\email{ssshimi@ipc.shizuoka.ac.jp}
\thanks{The work of SS was partially supported by JSPS Grant-in-Aid for Scientific Research (B) \#24340025 
and Challenging Exploratory Research \#23654048. 
The work of GS was partially supported by a grant from the Simons Foundation (\#245959) and by
a grant from NSF (DMS-1265579).}

\author[G.~Simonett]{Gieri Simonett}
\address
{Department of Mathematics\\
         Vanderbilt University \\
         Nashville, TN~37240, USA}
\email{gieri.simonett@vanderbilt.edu}

\author[M.~Wilke]{Mathias Wilke}
\address
{Institut f\"ur Mathematik \\
         Martin-Luther-Universit\"at Halle-Witten\-berg\\
         D-06099 Halle, Germany}
\email{mathias.wilke@mathematik.uni-halle.de}



\date{September 30, 2013}
\dedicatory{Dedicated to Professor Yoshihiro Shibata on the occasion of his 60th anniversary}

\begin{abstract}
Our study of the basic model for incompressible two-phase flows with phase transitions consistent with  thermodynamics \cite{PSSS12}, \cite{PrSh12}, \cite{PSW13}, \cite{PSZ12} is extended to the case of temperature-dependent surface tension. We prove well-posedness in an $L_p$-setting, study the stability of the equilibria of the problem, and show that a solution which does not develop singularities exists globally, and if its limit set contains a stable equilibrium it converges to this equilibrium in the natural state manifold for the problem as time goes to infinity.
\end{abstract}

\maketitle
\section{Introduction}
Let $\Omega\subset \R^{n}$ be a bounded domain of class $C^{3}$, $n\geq2$.
$\Omega$ contains two phases: at time $t$, phase $i$ occupies
subdomain $\Omega_i(t)$ of
$\Omega$. Assume $\partial \Omega_1(t)\cap\partial \Omega=\emptyset$; this means no
{\em boundary intersection}.
The closed compact hypersurface $\Gamma(t):=\partial \Omega_1(t)\subset \Omega$
forms the interface between the phases, which is allowed to be disconnected.
\goodbreak

We let
$u$ denote the velocity field,
$\pi$ the pressure field,
$T$ the stress tensor,  $D(u)=(\nabla u +[\nabla u]^{\sf T})/2$ the rate of strain tensor,
 $\theta$  the (absolute) temperature field,
 $\nu_\Gamma$ the outer normal  of $\Omega_1$,
 $u_\Gamma$ the velocity field of the interface,
 $V_\Gamma=u_\Gamma\cdot\nu_\Gamma$ the normal velocity of $\Gamma(t)$,
 $H_\Gamma=H(\Gamma(t))=-{\rm div}_\Gamma \nu_\Gamma$ the curvature of $\Gamma(t)$,
 $j_\Gamma$ the phase flux, and
 $[\![v]\!]=v_2-v_1$ the jump of the quantities $v_i\in C(\overline\Omega_i(t))$ across $\Gamma(t)$.

Several quantities are derived from the specific {\em free energy densities} $\psi_i(\theta)$ as follows. $\epsilon_i(\theta):= \psi_i(\theta)+\theta\eta_i(\theta)$ means the
specific internal energy density in phase $i$, $\eta_i(\theta) =-\psi_i^\prime(\theta)$ the density of the specific entropy,
 $\kappa_i(\theta):= \epsilon^\prime_i(\theta)>0$ the  heat capacity, and
$l(\theta)=\theta[\![\psi^\prime(\theta)]\!]=-\theta[\![\eta(\theta)]\!]$ the latent heat. Further $d_i(\theta)>0$ denotes the coefficient of heat conduction in Fourier's law, $\mu_i(\theta)>0$ the viscosity in Newton's law, $\rho_1,\rho_2>0$ the constant, positive densities of the phases,
  and $\sigma>0$ the (coefficient of) surface tension.
 In the sequel we drop the index $i$, as there is no danger of confusion; we just keep in mind that the coefficients depend on the phases.

In the previous papers \cite{PSW13} and \cite{PSZ12} we have mathematically analyzed the following problem with sharp interface:

\medskip

\noindent
Find a family of closed compact hypersurfaces $\{\Gamma(t)\}_{t\geq0}$ contained in $\Omega$
and appropriately smooth functions $u:\R_+\times \bar{\Omega} \to \R^n$, and $\pi,\theta:\R_+\times\bar{\Omega}\rightarrow\R$ such that
\begin{equation}
\hspace{1.8cm}
\begin{aligned}
\label{NS}
\rho(\partial_t u + u\cdot\nabla u) -{\rm div }\, T&=0  \quad && \mbox{in }\;  \Omega\setminus \Gamma(t),\\
T=2\mu(\theta)D(u)-\pi I,\quad {\rm div}\, u&=0 \quad && \mbox{in }\; \Omega\setminus \Gamma(t),\\
[\![\frac{1}{\rho}]\!]j_\Gamma^2\nu_\Gamma-[\![T\nu_\Gamma]\!] =\sigma H_\Gamma \nu_\Gamma,\quad
[\![u]\!]&= [\![\frac{1}{\rho}]\!] j_\Gamma \nu_\Gamma \quad && \mbox{on }\Gamma(t),\\
u=0 \;\mbox{ on }\; \partial\Omega,\quad u(0)&=u_0 \quad &&\mbox{in }\;\Omega.
\end{aligned}
\end{equation}
\\
\begin{equation}
\begin{aligned}
\label{Heat}
\rho\kappa(\theta)(\partial_t \theta  + u\cdot\nabla \theta)\!-\!{\rm div}\,(d(\theta)\nabla\theta)\!-\!2\mu|D(u)|_2^2
&=0 \quad &&\mbox{ in }\; \Omega\setminus \Gamma(t),\\
-l(\theta)j_\Gamma -[\![d(\theta)\partial_{\nu_\Gamma} \theta]\!]=0,  \quad
\mbox{}[\![\theta]\!]&=0\quad &&\mbox{ on }\;\Gamma(t),\\
\partial_\nu \theta =0\;\mbox{ on }\partial\Omega,\quad
\theta(0)&=\theta_0\quad &&\mbox{ in }\;\Omega.
\end{aligned}
\end{equation}
\\
\begin{equation}
\hspace{-4mm}
\begin{aligned}
\label{GTS}
[\![\psi(\theta)]\!] +[\![\frac{1}{2\rho^2}]\!]j_\Gamma^2 -[\![\frac{T\nu_\Gamma\cdot\nu_\Gamma}{\rho}]\!]&=0
\quad && \mbox{on }\Gamma(t),\\
V_\Gamma =u_\Gamma\cdot\nu_\Gamma &= u\cdot\nu_\Gamma -\frac{1}{\rho} j_\Gamma\quad &&\mbox{on }\Gamma(t),\\
\Gamma(0)&=\Gamma_0.
\end{aligned}
\end{equation}
This model is explained in more detail in \cite{PSSS12}; see also \cite{Gur07, IsTa06}. It is thermodynamically consistent in the sense that in absence of exterior forces and heat sources, the total mass  and the total energy are preserved, and the total entropy is nondecreasing. This model is in some sense the simplest non-trivial sharp interface model for incompressible Newtonian two-phase flows taking into account phase transitions driven by temperature.

Note that in this model neither {\em kinetic undercooling} nor {\em temperature dependence of the surface tension} $\sigma$, i.e.\ {\em Marangoni forces}, have been taken into account. It is the aim of the present paper to remove these shortcomings. Here we concentrate on the case of constant densities $\rho_i>0$ which are not equal, i.e.\ $[\![\rho]\!]\neq0$.
To achieve this goal, the model has to be adjusted carefully. Surface tension $\sigma(\theta_\Gamma)$ here is precisely the {\em free surface energy density}. Therefore we define, in analogy to the bulk case, the {\em surface energy density} $\epsilon_\Gamma$, the {\em surface entropy density} $\eta_\Gamma$, the {\em surface heat capacity} $\kappa_\Gamma$, and {\em surface latent heat} $l_\Gamma$ by means of the relations
\begin{align*}
\epsilon_\Gamma(\theta_\Gamma) = \sigma(\theta_\Gamma)+\theta_\Gamma\eta_\Gamma(\theta_\Gamma),\quad & \eta_\Gamma(\theta_\Gamma)=-\sigma^\prime(\theta_\Gamma)\\
\kappa_\Gamma(\theta_\Gamma) = \epsilon_\Gamma^\prime(\theta_\Gamma)=-\theta_\Gamma\sigma^{\prime\prime}(\theta_\Gamma),\quad &l_\Gamma(\theta_\Gamma)=\theta_\Gamma \sigma^\prime(\theta_\Gamma).
\end{align*}
Then {\em total surface energy} will be
$${\sf E}_\Gamma= \int_\Gamma \epsilon_\Gamma(\theta_\Gamma) d\Gamma,$$
and {\em total surface entropy} reads
$$\Phi_\Gamma =\int_\Gamma \eta_\Gamma(\theta_\Gamma) d\Gamma.$$
Note that in case $\sigma=const$ we have ${\sf E}_\Gamma = \sigma |\Gamma|$ and $\Phi_\Gamma=0$.
We point out that experiments have shown that in certain situations surface heat capacity cannot be neglected, see \cite{DaWa07}.
In case $\kappa_\Gamma$ is not identically zero (i.e.\ if $\sigma$ is not a linear function of $\theta_\Gamma$), there will also be a non-trivial {\em surface heat flux} $q_\Gamma$ which we assume to satisfy Fourier's law, that is
$$q_\Gamma=- d_\Gamma(\theta_\Gamma) \nabla_\Gamma \theta_\Gamma,$$ 
with coefficient of surface heat diffusivity $d_\Gamma(\theta_\Gamma)>0$.
In analogy to the bulk, the surface heat flux induces the contribution 
$\int_\Gamma -(q_\Gamma\cdot\nabla_\Gamma)/\theta_\Gamma^2\, d\Gamma$ to the production of surface entropy. {\em Kinetic undercooling} with coefficient $\gamma(\theta_\Gamma)>0$ produces surface entropy 
$\int_\Gamma \gamma(\theta)j_\Gamma^2/\theta_\Gamma\,d\Gamma$. Following the derivation in \cite{PSSS12}, this leads to the following three modifications of the system \eqref{NS}, \eqref{Heat}, \eqref{GTS}. The momentum balance on the interface becomes
$$[\![\frac{1}{\rho}]\!]j_\Gamma^2\nu_\Gamma-[\![T\nu_\Gamma]\!]
=\sigma(\theta_\Gamma) H_\Gamma \nu_\Gamma + \sigma^\prime(\theta_\Gamma)\nabla_\Gamma\theta_\Gamma.$$
The energy balance on the interface reads
$$\kappa_\Gamma(\theta_\Gamma)\frac{D}{Dt}\theta_\Gamma +{\rm div}_\Gamma q_\Gamma= [\![d(\theta)\partial_{\nu_\Gamma}\theta]\!]
+ l(\theta)j_\Gamma+l_\Gamma(\theta_\Gamma){\rm div}_\Gamma u_\Gamma+\gamma(\theta_\Gamma)j_\Gamma^2,$$
and the Gibbs-Thomson law changes to
$$ [\![\psi(\theta)]\!] + [\![\frac{1}{2\rho^2}]\!]j_\Gamma^2 -[\![\frac{T\nu_\Gamma\cdot\nu_\Gamma}{\rho}]\!]
=-\gamma(\theta_\Gamma) j_\Gamma.$$
Here $D/Dt$ denotes the Lagrangian derivative coming from the velocity $u_\Gamma$ of the interface, and we employ the symbol $P_\Gamma$ for the orthogonal projection onto the tangent bundle of $\Gamma$.
Note that $\theta_\Gamma=\theta{|_\Gamma}$ is the trace of $\theta$ on $\Gamma$ as $[\![\theta]\!]=0$. We assume the tangential part of $u$ to be continuous across $\Gamma$, i.e.\ $[\![P_\Gamma u]\!]=0$, and $P_\Gamma u_\Gamma =P_\Gamma u{|_\Gamma}$. 
Then the quantities $j_\Gamma$, $V_\Gamma$  and $u_\Gamma$ can be expressed as
$$ j_\Gamma=[\![u\cdot\nu_\Gamma]\!]/[\![1/\rho]\!],\quad 
V_\Gamma=[\![\rho u\cdot \nu_\Gamma]\!]/[\![\rho]\! ],\quad u_\Gamma = P_\Gamma u + V_\Gamma\nu_\Gamma.$$
The complete extended model now reads as follows.

\medskip

\noindent
In the bulk $\Omega\setminus\Gamma(t)$:
\begin{align}\label{bulk}
\rho(\partial_t u + u\cdot \nabla u) -2{\rm div}(\mu(\theta)D(u)) +\nabla \pi&=0,\nn\\
2D(u)=\nabla u +[\nabla u]^{\sf T},\quad {\rm div}\, u &=0,\\
\rho\kappa(\theta)(\partial_t\theta +u\cdot\nabla \theta)- {\rm div}(d(\theta)\nabla\theta)&=2\mu(\theta)|D(u)|_2^2.\nn
\end{align}
On the interface $\Gamma(t)$:
\begin{align}\label{interface}
&[\![P_\Gamma u]\!]=0,\; P_\Gamma u_\Gamma 
=P_\Gamma u{|_\Gamma},\; [\![u\cdot\nu_\Gamma]\!]=[\![1/\rho]\!]j_\Gamma,\:
[\![\theta]\!]=0,\; \theta_\Gamma = \theta{|_\Gamma}, \nn\\
&[\![1/\rho]\!]j_\Gamma^2\nu_\Gamma-2[\![\mu(\theta)D(u)\nu_\Gamma]\!] +[\![\pi]\!]\nu_\Gamma=\sigma(\theta_\Gamma)H_\Gamma\nu_\Gamma+\sigma^\prime(\theta_\Gamma)\nabla_\Gamma \theta_\Gamma,\nn\\
&\kappa_\Gamma\frac{D}{Dt}\theta_\Gamma -{\rm div}_\Gamma(d_\Gamma(\theta_\Gamma)\nabla_\Gamma\theta_\Gamma)=\\ 
&\hspace{2cm}=[\![d(\theta)\partial_\nu\theta]\!] + l(\theta)j_\Gamma
+\gamma(\theta_\Gamma)j_\Gamma^2+l_\Gamma(\theta_\Gamma){\rm div}_\Gamma u_\Gamma,\nn\\
&[\![\psi(\theta)]\!] +[\![1/2\rho^2]\!]j_\Gamma^2-2[\![\mu(\theta) D(u)\nu_\Gamma\cdot\nu_\Gamma/\rho]\!]+[\![\pi/\rho]\!]
=-\gamma(\theta_\Gamma)j_\Gamma,\nn \\
&V_\Gamma=u\cdot\nu_\Gamma-j_\Gamma/\rho.\nn
\end{align}
On the outer boundary $\partial\Omega$:
$$ u=0,\quad \partial_\nu \theta=0.$$
Initial conditions:
$$ \Gamma(0)=\Gamma_0,\quad u(0)=u_0, \quad \theta(0)=\theta_0.$$
This model has conservation of total mass and total energy, and total entropy is non-decreasing. Indeed, along smooth solutions we have
\begin{align*}
\frac{d}{dt}(\Phi_b(t)+\Phi_\Gamma(t))&= \int_{\Omega} [2\mu(\theta)|D(u)|_2^2/\theta +d(\theta)|\nabla\theta|^2/\theta^2]dx\\
&+ \int_\Gamma [d_\Gamma(\theta_\Gamma)|\nabla_\Gamma\theta_\Gamma|^2/\theta_\Gamma^2 + \gamma(\theta_\Gamma) j_\Gamma^2/\theta_\Gamma ]
 d\Gamma\geq0,
\end{align*}
where $\Phi_b=\int_\Omega \rho\eta(\theta) dx$.

For the sake of well-posedness we assume that $\psi_i$ and $\sigma$ are strictly concave. Experiments show that $\sigma$ is also strictly decreasing and positive for low temperatures. Therefore, $\sigma$ has precisely one zero $\theta_c>0$ which we call the {\em critical temperature}. As the problem in the range $\theta>\theta_c$ is no longer well-posed, we restrict our attention to the interval $\theta\in(0,\theta_c)$. In all of the paper we impose the following assumptions.

\bigskip

\noindent
{\bf a) Regularity.}
$$\mu_i,d_i,d_\Gamma,\gamma\in C^2(0,\theta_c),\quad \psi_i,\sigma\in C^3(0,\theta_c).$$
\noindent
{\bf b) Well-posedness.}
$$ \kappa_i,\kappa_\Gamma,\mu_i,d_i,d_\Gamma,\sigma>0,\; \gamma\geq0 \; \mbox{ in }\; (0,\theta_c),
\quad  0<\theta_0<\theta_c \; \mbox{ in }\; \bar{\Omega}.$$

\noindent
{\bf c) Compatibilities.}
$$ {\rm div}\,u_0=0 \mbox{ in } \Omega\setminus\Gamma_0,$$
$$2P_{\Gamma_0}[\![\mu(\theta_0)D(u)\nu_{\Gamma_0}]\!]+\sigma^\prime(\theta_0)\nabla_{\Gamma_0}\theta_0=0,
\quad P_{\Gamma_0}u_{\Gamma}(0)= P_{\Gamma_0}{u_0}{|_{\Gamma_0}}$$
$$P_{\Gamma_0}[\![u_0]\!]=0,\; [\![\theta_0]\!]=0,\; \theta_{\Gamma}(0)={\theta_{0}}{|_{\Gamma_0}},\;
{u_0}{|_{\partial\Omega}}=0,\; \partial_\nu {\theta_0}{|_{\partial\Omega}}=0.$$

\bigskip

\noindent
Below we present a rather complete analysis of problem \eqref{bulk}, \eqref{interface} which parallels that in \cite{PSW13} where the simpler case
$\sigma>0$ constant and $\gamma\equiv0$ has been studied.  We obtain local well-posedness of the problem in an $L_p$-setting, prove that the stability properties of equilibria are the same as in \cite{PSW13}, and we show that any bounded solution that does not develop singularities converges to an equilibrium as $t\to\infty$ in the state manifold $\cSM$ which is the same as in \cite{PSW13}.

\section{Total Entropy and Equilibria}
\noindent
{\bf (a)} As we have seen in Section 1, the total mass ${\sf M}=\int_{\Omega\setminus\Gamma}\rho\,dx $ and 
the total energy 
\begin{equation*}
{\sf E}={\sf E}(u,\theta,\theta_\Gamma,\Gamma)
:=\int_{\Omega\setminus\Gamma}\{(\rho/2)|u|^2 + \rho \eps(\theta)\}\,dx +\int_\Gamma \epsilon_\Gamma(\theta_\Gamma)\,d\Gamma
\end{equation*}
are conserved, and the total entropy
\begin{equation*}
\Phi=\Phi(\theta,\theta_\Gamma,\Gamma)
:=\int_{\Omega\setminus\Gamma} \rho\eta(\theta)\,dx + \int_\Gamma \eta_\Gamma(\theta_\Gamma)\,d\Gamma
\end{equation*}
is nondecreasing along smooth solutions. 
Even more,
$-\Phi$ is a strict Lyapunov functional in the sense that it is strictly decreasing along
 smooth solutions which are non-constant in time.
Indeed, if at some time
$t_0\ge 0$ we have
 \begin{equation*}
 \frac{d}{dt}\Phi(u(t_0),\Gamma(t_0)) =0,
 \end{equation*}
 then
 \begin{equation*}
 \int_\Omega[2\mu(\theta)|D(u)|^2/\theta+ d(\theta)|\nabla \theta|^2/\theta^2] dx+ \int_\Gamma [d_\Gamma(\theta_\Gamma)|\nabla \theta_\Gamma|^2/\theta^2+\gamma(\theta_\Gamma)j_\Gamma^2]\, d\Gamma = 0,
 \end{equation*}
which yields $D(u(t_0))=0$ and $\nabla \theta(t_0)=0$
in $\Omega$,  as well as $\nabla_\Gamma \theta_\Gamma(t_0)=0$ and $j_\Gamma(t_0)=0$ on $\Gamma(t_0)$. As in \cite{PSW13} this implies $u(t_0)=0$ and $\theta(t_0)=const=\theta_\Gamma(t_0)$ in $\Omega$. From the equations we see that 
$(V_\Gamma(t_0),j_\Gamma(t_0))=(0,0)$ and therefore $\pi$ is also constant in the components of the phases, and
\begin{equation}
\begin{aligned}
\label{eq-cond}
&[\![\pi]\!]= \sigma(\theta_\Gamma)H_\Gamma,\\
&[\![\psi(\theta_\Gamma)]\!] + [\![\pi/\rho]\!]=0.\nn
\end{aligned}
\end{equation}
These relations show that the curvature $H_\Gamma$ is constant over all of $\Gamma$, and it determines the values of the pressures in the phases, in particular $\pi$ is constant in each phase, not only in its components.
Since $\Omega$ is bounded, we may conclude that $\Gamma(t_0)$ is a union of finitely many, say $m$, disjoint spheres of equal radius,
i.e.\ $(u(t_0),\theta(t_0),\Gamma(t_0))$ is an equilibrium.
Therefore, the {\em limit sets} of solutions in the state manifold $\cSM_\Gamma$,  to be defined below, are contained in the
$(mn+2)$-dimensional manifold of equilibria
\begin{align}
\label{equilibria-I}
\cE&=\big\{\big(0,\theta_\ast,\bigcup_{1\le l\le m}S_{R_\ast}(x_l)\big):
\theta_\ast\in (0,\theta_c),
\bar B_{R_\ast}(x_l)\subset \Omega,\\
&\hspace{4cm} \bar{B}_{R_\ast}(x_l)\cap \bar{B}_{R_\ast}(x_k)=\emptyset, k\neq l\big\},\nonumber
\end{align}
where $S_{R_\ast}(x_l)$ denotes the sphere with radius $R_\ast$ and center $x_l$.
Here $R_\ast>0$ is uniquely determined by the total mass and by the number $m$ of spheres, and $\theta_*$ is uniquely given by the total energy.

\bigskip

\noindent
{\bf b)}\, Another interesting observation is the following. Consider the
critical points of the functional $\Phi(u,\theta,\theta_\Gamma,\Gamma)$ with constraint 
${\sf M}={\sf M}_0$, ${\sf E}(u,\theta,\theta_\Gamma,\Gamma)={\sf E}_0$,
say on
$$\{(u,\theta,\theta_\Gamma,\Gamma):\, (u,\theta)\in BUC(\bar{\Omega}\setminus\Gamma)^{n+1},
\Gamma\in \cMH^2(\Omega),\theta_\Gamma\in C(\Gamma)\},$$
 see below for the definition of
$\cMH^2(\Omega)$. So here we do not assume from the beginning that $\theta$ is continuous across $\Gamma$, and $\theta_\Gamma$ denotes surface temperature.
Then by the method of Lagrange multipliers, there
are constants $\lambda,\mu\in\R$ such that at a critical point $(u_*,\theta_*,\theta_{\Gamma*},\Gamma_*)$ we have
\begin{equation}
\label{VarEq} \Phi^\prime(u_*,\theta_*,\theta_{\Gamma*},\Gamma_*)+\lambda {\sf M}^\prime(\Gamma_*)+\mu {\sf E}^\prime(u_*,\theta_*,\theta_{\Gamma*},\Gamma_*)=0.
\end{equation}
The derivatives of the functionals are given by
\begin{align*}
\langle \Phi^\prime(u,\theta,\theta_\Gamma,\Gamma) |(v,\vartheta,\vartheta_\Gamma,h)\rangle
&= (\rho\eta^\prime(\theta)|\vartheta)_\Omega + (\eta^\prime_\Gamma(\theta_\Gamma)|\vartheta_\Gamma)_\Gamma\\ &-([\![\rho \eta(\theta)]\!]+\eta_\Gamma(\theta_\Gamma)H(\Gamma)|h)_{\Gamma},\\
\langle {\sf M}^\prime(\Gamma)|h\rangle &= -([\![\rho]\!]|h)_\Gamma,
\end{align*}
with $H(\Gamma):=H_\Gamma$, and
\begin{align*}
\langle {\sf E}^\prime(u,\theta,\theta_\Gamma,\Gamma) |(v,\vartheta,\vartheta_\Gamma,h)\rangle &= (\rho u|v)_\Omega+(\rho\epsilon^\prime(\theta)|\vartheta)_\Omega
+(\epsilon^\prime_\Gamma(\theta_\Gamma)|\vartheta_\Gamma)_\Gamma\\ 
&-([\![(\rho/2)|u|^2 + \rho\epsilon(\theta)]\!]+\epsilon_\Gamma(\theta_\Gamma) 
H(\Gamma)|h)_{\Gamma}.\end{align*}
Setting first $(u,\vartheta_\Gamma,h)=(0,0,0)$ and varying $\vartheta$ in \eqref{VarEq} we obtain
$$\eta^\prime(\theta_*) + \mu \epsilon^\prime(\theta_*)=0\quad \mbox{ in } \Omega,$$
 and similarly varying $\vartheta_\Gamma$ yields
$$\eta_\Gamma^\prime(\theta_{\Gamma*}) +\mu \epsilon_\Gamma^\prime(\theta_{\Gamma_*})=0\text{ on $\Gamma_*$}.$$
The relations $\eta(\theta)=-\psi^\prime(\theta)$ and
$\epsilon(\theta)=\psi(\theta)-\theta\psi^\prime(\theta)$ imply
$0=-\psi^{\prime\prime}(\theta_*)(1+\mu \theta_*)$, and this shows that
$\theta_*=-1/\mu$ is constant in $\Omega$, since
$\kappa(\theta)=-\theta\psi^{\prime\prime}(\theta)>0$ for all $\theta\in(0,\theta_c)$ by assumption.
Similarly on $\Gamma_*$ we obtain $\theta_{\Gamma_*}=-1/\mu$ constant as well, provided $\kappa_\Gamma(\theta_\Gamma)>0$, hence in particular $\theta_*\equiv \theta_{\Gamma_*}$.

Next varying $v$ implies $\mu u_*=0$, and hence $u_*=0$ as $\mu\neq 0$. 
Finally, varying $h$ we get
$$-([\![\rho\eta(\theta_*)]\!] +\eta_\Gamma(\theta_{\Gamma_*})H_\Gamma)-\lambda[\![\rho]\!]
-\mu([\![\rho\epsilon(\theta_*)]\!] + \epsilon_\Gamma(\theta_{\Gamma_*}) H(\Gamma_*))=0\text{ on $\Gamma_*$}.$$
This implies with the above relations
$$[\![\rho\psi(\theta_*)]\!]+\sigma(\theta_*) H(\Gamma_*)=\lambda[\![\rho]\!]\theta_{\Gamma_*}.$$
 Since $\theta_*$ is constant and assuming $\theta_*\in (0,\theta_c)$,  we see with $\sigma_*>0$
that $H(\Gamma_*)$ is constant.
Therefore, as $\Omega$ is bounded, $\Gamma_*$ is a sphere
whenever connected, and a union of finitely many disjoint spheres of
equal size otherwise. Thus the critical points of the entropy
functional for prescribed energy are precisely the equilibria of the
problem \eqref{bulk}, \eqref{interface}.

\bigskip

\noindent
{\bf(c)}\, Going further, suppose we have an equilibrium $e_*:=(0,\theta_*,\theta_{\Gamma_*},\Gamma_*)$
where the total entropy has a local maximum w.r.t.\ the constraints
${\sf M}={\sf M}_0$ and ${\sf E}={\sf E}_0$.
Then $\cD_*:=[\Phi+\lambda {\sf M}+\mu {\sf E}]^{\prime\prime}(e_*)$ is negative semi-definite on the kernel of
$({\sf M}^\prime,{\sf E}^\prime)(e_*)$,
where $\lambda$ and $\mu$  are the fixed Lagrange
multipliers found above. The kernel of ${\sf M}^\prime(e)$ is given by $(1|h)_\Gamma=0$, as $[\![\rho]\!]\neq0$, and that of ${\sf E}^\prime(e)$ is determined
by the identity
\begin{equation*}
(\rho u|v)_\Omega+(\rho\kappa(\theta)| \vartheta)_\Omega + (\kappa_\Gamma(\theta_\Gamma)|\vartheta_\Gamma)_\Gamma-([\![\rho\epsilon(\theta)]\!] + \epsilon_\Gamma(\theta_\Gamma) H(\Gamma)| h)_\Gamma=0,
\end{equation*}
which at equilibrium yields
\begin{equation}
\label{kE}
(\rho\kappa_\ast| \vartheta)_\Omega + (\kappa_{\Gamma_*}|\vartheta_\Gamma)_\Gamma=0,
\end{equation}
where $\kappa_\ast :=\kappa(\theta_\ast)$, $\kappa_{\Gamma_*}:=\kappa_\Gamma(\theta_*)$, and $\sigma_*=\sigma(\theta_{\Gamma_*})$.
On the other hand, a straightforward calculation yields with
$z=(v,\vartheta,\vartheta_\Gamma,h)$
\begin{align}
\label{2var}
-\langle \cD_* z|z\rangle &=(\rho v|v)_\Omega+\frac{1}{\theta_*}\big[(\rho\kappa_\ast \vartheta|\vartheta)_\Omega + (\kappa_{\Gamma*}\vartheta_\Gamma|\vartheta_\Gamma)_\Gamma
- \sigma_* \theta_*( H^\prime(\Gamma_*) h|h)_\Gamma\big].
\end{align}
As  $\kappa_\ast$ and $\kappa_{\Gamma_*}$ are positive, we see that the form
$\langle \cD z|z\rangle$ is negative semi-definite  as soon as
$H^\prime ({\Gamma_*}) $ is negative semi-definite.
 We have
\begin{equation*}
H^\prime(\Gamma_\ast) = (n-1)/{R^2_\ast} + \Delta_{\Gamma_\ast},
\end{equation*}
where $\Delta_{\Gamma_\ast}$ denotes the
Laplace-Beltrami operator on $\Gamma_\ast$ and $R_\ast$ means the radius of
an equilibrium sphere. To derive necessary conditions for an equilibrium $e_*$ to be a
local maximum of entropy, we
suppose that $\Gamma_\ast$ is not connected,
i.e. $\Gamma_\ast$ is a finite union of spheres $\Gamma^k_\ast$.
Set $\vartheta=\vartheta_\Gamma=0$, and let $h=h_k$ be
constant on $\Gamma^k_\ast$
with $\sum_k h_k=0$.
Then the constraint \eqref{kE} holds,
and with the volume $\omega_n$ of the unit sphere in $\R^n$
\begin{equation*}
\langle \cD z|z\rangle= (\sigma_*\theta_\ast)((n-1)/R^2_\ast)\omega_n R^{n-1}_* \,\sum_k h_k^2 >0,\
\end{equation*}
hence $\cD$ cannot be negative semi-definite in this case, as $\sigma_*>0$. Thus if $e_\ast$
is an equilibrium with locally maximal total entropy, then $\Gamma_\ast$ must be
connected, and hence both phases are connected.

On the other hand, if $\Gamma_*$ is connected then $H^\prime(\Gamma_*)$ is negative semi-definite
on functions with mean zero, hence in this case $\cD$ is in fact positive semi-definite.
We will see below that connectedness of $\Gamma_*$ is precisely the condition for stability of the equilibrium $e_*$.

\bigskip

\noindent
{\bf(d)}\, Summarizing, we have shown
\begin{itemize}
\item The total energy is constant along smooth solutions of \eqref{bulk},\eqref{interface}.
\item The negative total entropy is a strict Ljapunov functional for \eqref{bulk},\eqref{interface}.
\item The equilibria of \eqref{bulk},\eqref{interface} are precisely the critical points
of the entropy functional with prescribed total energy and total mass.
\item
If the entropy functional with prescribed energy and total mass
has a local maximum at $e_*=(0,\theta_*,\theta_{\Gamma_*},\Gamma_*)$, then $\Gamma_*$ is connected.
\end{itemize}

\medskip


\noindent
It should be noted that we are using the term {\em equilibrium} to describe a {\em stationary solution} of the system,
while {\em a thermodynamic equilibrium} would - by definition - require the entropy production to be zero.
Our results show that stable equilibria are precisely those that guarantee the system to be in a thermodynamic equilibrium.

\section{Local Well-posedness and the Semiflow} 
$\theta_\Gamma$ is not a real system variable as it is the trace of $\theta$ on the interface.
For analytical reasons it is a useful quantity while we consider the linear problem and the nonlinear problem
on the reference manifold $\Sigma$ after a Hanzawa transform.
Local well-posedness of Problem \eqref{bulk},\eqref{interface} is based on maximal $L_p$-regularity
of its principal linearization and on the contraction mapping principle.

\subsection{Principal Linearization}

The principal part of the linearized problem reads as follows
\begin{equation}
\hspace{-2mm}
\begin{aligned}
\label{linNS}
\rho\partial_t u -\mu_0(x) \Delta u +\nabla \pi&=\rho f_u &&\mbox{in }\; \Omega\setminus\Sigma,\\
{\rm div}\, u&=g_d && \mbox{in }\; \Omega\setminus\Sigma, \\
P_\Sigma[\![u]\!]+c(t,x)\nabla_\Sigma h &=P_\Sigma g_u  && \mbox{on } \Sigma,\\
\!\!-2[\![\mu_0(x) D(u) \nu_\Sigma]\!] \!+\! [\![\pi]\!]\nu_\Sigma\!-\!\sigma_0(x) \Delta_\Sigma h \nu_\Sigma
\!-\!\sigma_1(x) \nabla_\Sigma \theta_\Sigma  &= g  &&\mbox{on } \Sigma,\\
u&=0   &&\mbox{on } \partial\Omega,\\
u(0)&=u_0 &&\mbox{in }\; \Omega.
\end{aligned}
\end{equation}
\begin{equation}
\hspace{2.5cm}
\begin{aligned}
\label{linHeat}
\rho\kappa(x)\partial_t \theta -d(x)\Delta\theta &=\rho\kappa(x)f_\theta && \mbox{in }\; \Omega\setminus\Sigma,\\
\kappa_{\Gamma0}(x)\partial_t \theta_\Sigma -d_{\Gamma0}(x)\Delta_\Sigma\theta_\Sigma
 &=\kappa_{\Gamma0}(x) g_\theta  &&\mbox{on } \Sigma,\\
[\![\theta]\!]=0,\quad \theta_\Sigma&=\theta|_{\Sigma}&& \mbox{on } \Sigma,\\
\partial_\nu \theta&=0 && \mbox{on } \partial\Omega, \\
\theta(0)=\theta_0\quad \mbox{in }\; \Omega,\qquad \theta_\Sigma(0)&=\theta_0|_\Sigma && \mbox{on } \Sigma,\\
\end{aligned}
\end{equation}
\begin{equation}
\hspace{2cm}
\begin{aligned}
\label{linGTS}
[\![\rho]\!]\partial_t h-[\![\rho u\cdot\nu_\Sigma]\!]+b(t,x)\cdot\nabla_\Sigma h&= [\![\rho]\!]f_h &&\mbox{on } \Sigma,\\
-2[\![(\mu_0(x)/\rho) D(u) \nu_\Sigma\cdot\nu_\Sigma]\!]+[\![\pi/\rho]\!]&= g_h &&\mbox{on } \Sigma,\\
h(0)&=h_0 &&\mbox{on } \Sigma.
\end{aligned}
\end{equation}
Here  $\mu_0,\kappa_0$, $d_0$, $\kappa_{\Gamma0}$, $d_{\Gamma0}$, $\sigma_0$ and $\sigma_1$ are functions of $x$,
which are realized by $\mu_0(x)=\mu(\theta_0(x))$ and so on. 
The difference between the linear problem for variable surface tension
\eqref{linNS} \eqref{linHeat} \eqref{linGTS} and the linear problem for constant surface tension \cite[Section 3.1]{PSW13} is
the 4th equation of \eqref{linNS} and the 2nd equation of \eqref{linHeat}.

Observe that \eqref{linHeat} decouples from the remaining problems \eqref{linNS}-\eqref{linGTS}. 
Maximal $L_p$-regularity of
\eqref{linHeat} has been proved in \cite{PSW11}, while
maximal $L_p$-regularity of \eqref{linNS} and \eqref{linGTS} has been considered in \cite{PSW13}.
Therefore we obtain the maximal $L_p$-regularity of \eqref{linNS},\eqref{linHeat},\eqref{linGTS}.

\begin{thm}
\label{th:3.1}
Let $p>n+2$, $\rho_i>0$, $\rho_2\neq\rho_1$, $\mu_{0i}, \kappa_{0i}, d_{0i} \in C(\bar{\Omega}_i)$,
$\kappa_{\Gamma0}$, $d_{\Gamma0}$, $\sigma_0$, $\sigma_1\in C(\Sigma)$,
$\mu_{0i}, \kappa_{0i}, d_{0i}>0$,
$\kappa_{\Gamma0}, d_{\Gamma0}, \sigma_0, \sigma_1>0$, $i=1,2$,
$$(b,c)\in W^{1-1/2p}_p(J;L_p(\Sigma))^{n+1}\cap L_p(J;W^{2-1/p}_p(\Sigma))^{n+1},$$ 
where $J=[0,a]$.
Then the coupled system \eqref{linNS}, \eqref{linHeat}, \eqref{linGTS}  admits a unique solution $(u,\pi,\theta,\theta_\Sigma, h)$ with regularity
\begin{equation*}
\label{reg}
\begin{split}
&(u,\theta)\in H^1_p(J;L_p(\Omega))^{n+1}
  \cap L_p(J;H^2_p(\Omega\setminus\Sigma))^{n+1}=: \EE_{u,\theta}(J), \\
& [\![u\cdot\nu_\Sigma]\!]\in H^1_p(J;\dot{W}^{-1/p}_p(\Sigma)),\quad
\pi\in L_p(J;\dot{H}^1_p(\Omega\setminus\Sigma))=: \EE_{\pi}(J), \\
&\theta_\Sigma\in H^1_p(J; W_p^{-1/p}(\Sigma))\cap L_p(J; W_p^{2-1/p}(\Sigma))=: \EE_{tr\theta}(J), \\
&\pi_i:=\pi_{|_{\partial\Omega_i}}\in W^{1/2-1/2p}_p(J;L_p(\Sigma))\cap L_p(J;W^{1-1/p}_p(\Sigma))=: \EE_{tr\pi}(J),\; i=1,2,\\
& h\in W^{2-1/2p}_p(J;L_p(\Sigma))\cap H^1_p(J;W^{2-1/p}_p(\Sigma))
\cap L_p(J;W^{3-1/p}_p(\Sigma))=: \EE_{h}(J),\end{split}
\end{equation*}
if and only if the data
$(f_u, f_\theta, g_\theta, g_d,P_\Sigma g_u, g, f_h, g_h,u_0, \theta_0, \theta_0|_\Sigma, h_0)$
satisfy the following regularity 
\begin{itemize}
\item[(a)]
$(f_u,f_\theta) \in L_p(J;L_p(\Omega))^{n+1}$,
\vspace{1mm}
\item[(b)]
$g_\theta \in L_p(J; W_p^{-1/p}(\Sigma))$,
\vspace{1mm}
\item[(c)]
$g_d\in H^1_p(J; \dot{H}^{-1}_p(\Omega))\cap L_p(J; H^1_p(\Omega\setminus\Sigma))$,
\vspace{1mm}
\item[(d)]
$(g,g_h)\in W^{1/2-1/2p}_p(J;L_p(\Sigma))^{n+1}
\cap L_p(J;W^{1-1/p}_p(\Sigma))^{n+1}$,
\vspace{1mm}
\item[(e)]
$(P_\Sigma g_u,f_h)\in W^{1-1/2p}_p(J;L_p(\Sigma))^{n}\cap L_p(J;W^{2-1/p}_p(\Sigma))^{n}$,
\vspace{1mm}
\item[(f)]
$(u_0,\theta_0,\theta_0|_\Sigma,h_0)\in X_\gamma:= W_p^{2-2/p}(\Omega\setminus\Sigma)^{n+1}\times W_p^{2-3/p}(\Sigma)
\times W_p^{3-2/p}(\Sigma)$,
\end{itemize}
and compatibility conditions:
\begin{itemize}
\item[(g)]
${\rm div}\, u_0=g_d(0)$ in $\,\Omega\setminus\Sigma$,
\vspace{1mm}
\item[(h)]
$P_\Sigma[\![u_0]\!]+ c(0,\cdot)\nabla_\Sigma h_0=P_\Sigma g_u(0)$ on $\Sigma$,
\vspace{1mm}
\item[(i)]
$-P_\Sigma[\![\mu_0(\cdot)(\nabla u_0+[\nabla u_0]^{\sf T})\nu_\Sigma]\!] - \sigma_1(\cdot)\nabla_\Sigma\theta_\Sigma =P_\Sigma g(0)$ on
$\,\Sigma$. 
\end{itemize}
The solution map 
$$[(f_u, f_\theta, g_\theta, g_d,P_\Sigma g_u, g, f_h, g_h, u_0, \theta_0, \theta_0|_\Sigma, h_0)
\mapsto (u,\pi,\theta, \theta_\Sigma, h)]$$ is continuous between the corresponding spaces.
\end{thm}
\begin{rem} $X_\gamma$ is the time trace space of the solution space $\EE(J)$ with
\vspace{-0.7mm}
$$
\EE(J):=\EE_{u}(J)\times\EE_{\theta}(J)\times \EE_{tr\theta}(J)\times \EE_{h}(J).
$$
\end{rem}
\medskip
\subsection{ Local Existence}

As in \cite{PSW13}, the basic result for local well-posedness of Problem \eqref{bulk},\eqref{interface}
in an $L_p$-setting is the following theorem, which is proved by the contraction mapping principle.

\bigskip

\begin{thm} \label{wellposed} Let $p>n+2$, $\rho_1,\rho_2>0$, $\rho_1\neq\rho_2$.
We assume the conditions {\rm{\bf a), b), c)}} in Section 1 and the regularity conditions
$$
(u_0,\theta_0)\in W^{2-2/p}_p(\Omega\setminus\Gamma_0)^{n+1},
\quad\Gamma_0\in W^{3-2/p}_p.
$$
Then there exists a unique $L_p$-solution of Problem \eqref{bulk},\eqref{interface}  on
some possibly small but nontrivial time interval $J=[0,\tau]$.
\end{thm}

\bigskip

\noindent
Here the notation $\Gamma_0\in W^{3-2/p}_p$ means that $\Gamma_0$ is a $C^2$-manifold, such that
its (outer) normal field $\nu_{\Gamma_0}$ is of class $W^{2-2/p}_p(\Gamma_0)$. Therefore, the Weingarten tensor
 $L_{\Gamma_0}=-\nabla_{\Gamma_0}\nu_{\Gamma_0}$ of $\Gamma_0$ belongs to $W^{1-2/p}_p(\Gamma_0)$ which embeds into
$C^{\alpha+1/p}(\Gamma_0)$, with $\alpha=1-(n+2)/p>0$ since $p>n+2$ by assumption.
For the same reason we also have $u_0\in C^{1+\alpha}(\bar{\Omega}_i(0))^n$,
$\theta_0\in C^{1+\alpha}(\bar{\Omega}_i(0))$, $i=1,2$,
and $V_0\in C^{1+\alpha}(\Gamma_0)$.
The notion $L_p$-solution means that $(u,\pi,\theta,\theta_\Gamma, \Gamma)$ is obtained as the push-forward
of an $L_p$-solution 
$(\bar{u},\bar{\pi},\bar{\theta}, \bar{\theta}_\Sigma, h)$ of the transformed problem,
which means that $(\bar{u},\bar{\theta},\bar{\theta}_\Gamma, h)$ belongs to
$\EE(J)$. The regularity of the pressure is obtained from the equations.
\subsection{Time-Weights}\label{subsec:timeweights}

For later use we need an extension of the local existence result to spaces with time weights. For this purpose, given a UMD-Banach space $Y$
and $\mu\in(1/p,1]$, we define for $J=(0,t_0)$
$$K^s_{p,\mu}(J;Y):=\{u\in L_{p,loc}(J;Y): \; t^{1-\mu}u\in K^s_p(J;Y)\},$$
where $s\geq0$ and $K\in\{H,W\}$. It has been shown in \cite{PrSi04} that the operator $d/dt$ in $L_{p,\mu}(J;Y)$ with domain
$$\cD(d/dt)={_0H}^1_{p,\mu}(J;Y)=\{u\in H^1_{p,\mu}(J;Y):\; u(0)=0\}$$
is sectorial and admits an $H^\infty$-calculus with angle $\pi/2$.
This is the main tool to extend Theorem \ref{wellposed} to the time weighted setting, where the solution space $\EE(J)$ is replaced by
$\EE_\mu(J)$, and
$$
z\in \EE_\mu(J) \Leftrightarrow t^{1-\mu}z\in \EE(J).
$$
The trace spaces for $(u,\theta,h)$ for $p>3$ are then given by
\begin{align}\label{tracesp-mu}
&(u_0,\theta_0)\in W^{2\mu-2/p}_p(\Omega\setminus\Sigma)^{n+1},\;\theta_0|_\Sigma\in W^{2\mu-3/p}_p(\Sigma),
\;  h_0\in W^{2+\mu-2/p}_p(\Sigma),\nonumber\\
& h_1:=\partial_th{|_{t=0}}\in W^{2\mu-3/p}_p(\Sigma),
\end{align}
where for the last trace  we need in addition $\mu>3/2p$. Note that the embeddings
$$
\EE_{\mu,u,\theta}(J)\hookrightarrow  C(J;C^1(\bar{\Omega}_i))^{n+1},
\quad\EE_{\mu,h}(J)\hookrightarrow C(J;C^{2+\alpha}(\Sigma))\cap C^1(J; C^1(\Sigma))
$$
with $\alpha=1/2-n/p>0$ require $\mu>1/2+(n+2)/2p$, which is feasible since $p>n+2$ by assumption.
This restriction is needed for the estimation of the nonlinearities.

For these time weighted spaces we have the following result.

\begin{cor} \label{wellposed3} Let $p>n+2$, $\mu\in (1/2+(n+2)/2p,1]$, $\rho_1,\rho_2>0$, $\rho_1\neq\rho_2$.
We assume that the conditions {\rm{\bf a), b), c)}} in Section 1 and the regularity conditions
$$
(u_0,\theta_0)\in W^{2\mu-2/p}_p(\Omega\setminus\Gamma_0)^{n+1},
\quad\Gamma_0\in W^{2+\mu-2/p}_p
$$
are satisfied. Then the transformed problem admits a unique solution 
$$z=(u,\theta,\theta_\Sigma, h)\in \EE_\mu(0,\tau)$$
for some nontrivial time interval $J=[0,\tau]$. The solution depends continuously on the data.
For each $\delta>0$ the solution belongs to $\EE(\delta,\tau)$, i.e.\ it regularizes instantly.
\end{cor}

\section{Linear Stability of Equilibria} 
{\bf 1.} \, We call an equilibrium {\em non-degenerate} if the balls making up $\Omega_1(t)$
do neither touch each other nor the outer boundary;
this set is denoted by $\cE$. To derive the full linearization at a non-degenerate equilibrium
$e_*:=(0,\theta_*,\theta_{\Sigma*}, \Sigma)\in\cE$,  note that the quadratic terms $u\cdot\nabla u$, $u\cdot\nabla \theta$,
$|D(u)|_2^2$, $[\![u]\!]j_\Gamma$, and $j_\Gamma^2$ 
give no contribution to the linearization. Therefore we obtain the following fully linearized problem
for $(u, \pi, h)$, the relative temperature $\vartheta=(\theta-\theta_*)/\theta_*$
and $\vartheta_\Sigma=\vartheta|_\Sigma$.
\begin{equation}
\begin{aligned}
\label{elin-u}
\rho\partial_t u-\mu_* \Delta u +\nabla\pi &=\rho f_u\quad &&\mbox{in }\;\Omega\setminus\Sigma,\\
{\rm div}\, u &=g_d &&\mbox{in }\;\Omega\setminus\Sigma,\\
P_\Sigma[\![u]\!]&=P_\Sigma g_u &&\mbox{on } \Sigma,\\
 -2P_\Sigma [\![\mu_*D(u)\nu_\Sigma]\!] - \theta_*{\sigma^\prime_*}\nabla_\Sigma \vartheta_\Sigma
 &=P_\Sigma g && \mbox{on } \Sigma,\\
 -2[\![\mu_*D(u) \nu_\Sigma\cdot\nu_\Sigma]\!] +\lj\pi\rj +\sigma_* \cA_\Sigma h
-\theta_*{\sigma^\prime_*}H_*\vartheta_\Sigma
 &=g\cdot \nu_\Sigma && \mbox{on } \Sigma,\\
u&=0\quad &&\mbox{on } \partial\Omega,\\
u&=u_0\quad &&\mbox{in }\; \Omega,
\end{aligned}
\end{equation}
\\
\begin{equation}
\hspace{0mm}
\begin{aligned}
\label{elin-theta}
\rho\kappa_*\partial_t\vartheta -d_*\Delta \vartheta &=\rho\kappa_*f_\theta &&\mbox{in }\;
\Omega\setminus\Sigma, \hspace{-5mm}\\
[\![\vartheta]\!]&=0   &&\mbox{on } \Sigma,\\
\!\kappa_{\Gamma_*}\pd_t\vartheta_\Sigma \!-\!d_{\Gamma_*}\Delta_\Sigma\vartheta_\Sigma
\!-\!(l_*/\theta_*)j_\Sigma \!-\![\![d_*\partial_{\nu_\Sigma}\vartheta]\!]\!-\! {\sigma'_*} \text{div}_\Sigma u_\Sigma
&= \kappa_{\Gamma_*} g_\theta  &&\mbox{on } \Sigma, \\
\partial_\nu\vartheta&=0 &&\mbox{on } \partial\Omega,\\
\vartheta &=\vartheta_0  &&\mbox{in }\; \Omega,
\end{aligned}
\end{equation}
\\
\begin{equation}
\hspace{1mm}
\begin{aligned}
\label{elin-h}
-2[\![\mu_*D(u)\nu_\Sigma\cdot\nu_\Sigma/\rho]\!] +\lj \pi/\rho\rj
+ l_*\vartheta_\Sigma +\gamma_* j_\Sigma &= g_h && \mbox{on } \Sigma,\\
\partial_t h -[\![\rho u\cdot\nu_\Sigma]\!]/[\![\rho]\!] &= f_h && \mbox{on }\Sigma,\\
h(0)&=h_0 &&\mbox{on }\Sigma,
\end{aligned}
\end{equation}
where $\mu_*=\mu(\theta_*)$, $\kappa_*=\kappa(\theta_*)$, $d_*=d(\theta_*)$,
$\sigma_*=\sigma(\theta_*)$, $l_*=l(\theta_*)$, $\gamma_*=\gamma(\theta_*)$,
$l_{\Gamma_*}=l_\Gamma(\theta_*)$,
$\kappa_{\Gamma_*}=\kappa_\Gamma(\theta_*)$, $d_{\Gamma_*}=d_\Gamma(\theta_*)$, and
$$\cA_\Sigma =-H^\prime(0)=-(n-1)/R_*^2 -\Delta_\Sigma,\quad H_*=(n-1)/R_*.$$
Here we used that $l_{\Gamma_*}/\theta_*={\sigma^\prime_*}=\sigma'(\theta_*)$.
Finally, $u_\Sigma$ denotes the transformed velocity field $u_\Gamma$, and
$j_\Sigma$ is given by
$$
j_\Sigma:=[\![u\cdot\nu_\Sigma]\!]/[\![1/\rho]\!].
$$
The time-trace space $\EE_\gamma$ of $\EE(J)$ is given by
$$
(u_0,\vartheta_0,\vartheta_0|_\Sigma,h_0)\in\EE_\gamma
= [W^{2-2/p}_p(\Omega\setminus\Sigma)]^{n+1}\times W^{2-3/p}_p(\Sigma)
\times W^{3-2/p}_p(\Sigma),
$$
and the space of right hand sides is
\begin{align*}
((f_u,f_\theta),g_d, (f_h,P_\Sigma g_u),&(g,g_\theta,g_h)) \in \FF(J)\\
& := \FF_{u,\theta}(J)\times \FF_d(J)\times \FF_h(J)^{n+1}\times\FF_\theta(J)^{n+2},
\end{align*}
where
$$\FF_{u,\theta}(J)=L_p(J\times\Omega)^{n+1},\quad \FF_d(J)=H^1_p(J;\dot{H}^{-1}_p(\Omega))\cap L_p(J;H^1_p(\Omega)),$$
and
$$\FF_\theta(J)=W^{1/2-1/2p}_p(J;L_p(\Sigma))\cap L_p(J;W^{1-1/p}_p(\Sigma)),$$
$$\FF_h(J)=W^{1-1/2p}_p(J;L_p(\Sigma))\cap L_p(J;W^{2-1/p}_p(\Sigma)).$$
As the terms $(l_*/\theta_*)j_\Sigma$ and ${\sigma^\prime_*}{\rm div}_\Sigma u_\Sigma$ are lower order, the remaining system
is triangular, where the equations for $\theta$ decouple. Therefore, as in Section 3, it follows
from the maximal regularity results
in \cite{DPZ08,PrSh12,PSW13} and a standard perturbation argument that the operator $\LL$ defined by the left hand side of
\eqref{elin-u}, \eqref{elin-theta}, \eqref{elin-h} is an isomorphism from $\EE$ into $\FF\times\EE_\gamma$.
The range of $\LL$ is determined by the natural compatibility conditions.
If the time derivatives $\partial_t$ are replaced by $\partial_t +\omega$, $\omega>0$
sufficiently large, then this result is also true for $J=\R_+$.

\medskip

\noindent
{\bf 2.}\, We introduce a functional analytic setting as follows. Set
$$
X_0=L_{p,\sigma}(\Omega)\times L_p(\Omega)\times W^{-1/p}_p(\Sigma)\times W^{2-1/p}_p(\Sigma),
$$
where the subscript $\sigma$ means solenoidal, and define the operator $L$ by
\begin{align*}
&L(u, \vartheta,\vartheta_\Sigma, h)=\\
&
\big(-(\mu_*/\rho)\Delta u +\nabla\pi/\rho, -(d_*/\rho\kappa_*)\Delta \vartheta, \\
&-(1/{\kappa_{\Gamma_*}})( d_{\Gamma_*}\Delta_\Sigma\vartheta_\Sigma
+(l_*/\theta_*)j_\Sigma + [\![d_*\partial_{\nu_\Sigma}\vartheta]\!] + {\sigma^\prime_*}\text{div}_\Sigma u_\Sigma),
-[\![\rho u\cdot\nu_\Sigma]\!]/[\![\rho]\!]\big).
\end{align*}
To define the domain $\cD(L)$ of $L$, we set
\begin{align*}
X_1= \{& (u,\vartheta,\vartheta_\Sigma, h)\in H^2_p(\Omega\setminus\Sigma)^{n+1}\times W^{2-1/p}_p(\Sigma)
\times W^{3-1/p}_p(\Sigma): \\
& {\rm div}\, u=0\; \mbox{ in }\; \Omega\setminus\Sigma,\;
P_\Sigma[\![u]\!]=0,\;[\![\vartheta]\!]=0\; \mbox{ on } \; \Sigma, \\
&u=0,\;\partial_\nu\vartheta=0\;\mbox{ on }\;\partial\Omega\},
\end{align*}
and
\begin{align*}
\cD(L)= \{(u,\vartheta,\vartheta_\Sigma, h)\in X_1:  2 P_\Sigma[\![\mu_*D(u)\nu_\Sigma]\!]+
\theta_*{\sigma^\prime_*}\nabla_\Sigma \vartheta_\Sigma=0\; \mbox{ on } \, \Sigma\}.
\end{align*}
$\pi$ is determined as the solution of the weak transmission problem
\begin{eqnarray*}
&&(\nabla\pi|\nabla\phi/\rho)_2=((\mu_*/\rho)\Delta u|\nabla \phi)_2,\quad
\phi\in {H}^1_{p^\prime}(\Omega),\; \phi=0 \mbox{ on } \Sigma,\\
&& [\![\pi]\!]=-\sigma_* \cA_\Sigma h+\theta_*{\sigma^\prime_*}\vartheta_\Sigma H_*
+ 2[\![\mu_* (D(u)\nu_\Sigma|\nu_\Sigma)]\!],\quad \mbox{ on } \Sigma,\\
&& [\![\pi/\rho]\!]= 2[\![(\mu_*/\rho)(D(u)\nu_\Sigma|\nu_\Sigma)]\!]-l_*\vartheta
-\gamma_* [\![u\cdot\nu_\Sigma]\!]/[\![1/\rho]\!] \quad \mbox{ on } \Sigma.
\end{eqnarray*}
Let us introduce solution operators $T_k$, $k\in\{1,2,3\}$, as follows
\begin{align*}
\frac{1}{\rho}\nabla\pi
&=T_1((\mu_*/\rho)\Delta u)+T_2(-\sigma_* \cA_\Sigma h+\theta_*{\sigma^\prime_*}\vartheta_\Sigma H_*
+2[\![\mu_* (D(u)\nu_\Sigma|\nu_\Sigma)]\!])\\
&+T_3(2[\![(\mu_*/\rho)(D(u)\nu_\Sigma|\nu_\Sigma)]\!]-l_*\vartheta-\gamma_* [\![u\cdot\nu_\Sigma]\!]/[\![1/\rho]\!]).
\end{align*}
We refer to K\"ohne, Pr\"uss and Wilke \cite{KPW10} for the analysis of such transmission problems.
The linearized problem can be rewritten as an abstract evolution problem in $X_0$.
\begin{equation}\label{alp} \dot{z} + Lz =f,\quad t>0,\quad z(0)=z_0,\end{equation}
where $z=(u,\vartheta, \vartheta_\Sigma, h)$, $f=(f_u,f_\theta, f_\Sigma, f_h)$,
$z_0=(u_0,\vartheta_0, \vartheta_0|_\Sigma, h_0)$, provided $(g_d,g_u,g,g_\theta,g_h)=0$.
The linearized problem has maximal $L_p$-regularity, hence (\ref{alp}) has this property as well.
Therefore, by a well-known result, $-L$ generates an analytic $C_0$-semigroup in $X_0$;
see for instance  Proposition 1.1 in \cite{Pru03}. 

Since the embedding $X_1\hookrightarrow X_0$ is compact, the semigroup $e^{-Lt}$ as well as the 
resolvent $(\lambda+L)^{-1}$ of $-L$ are compact as well.
Therefore, the spectrum $\sigma(L)$ of $L$ consists of countably many eigenvalues of finite algebraic multiplicity, 
and it is independent of $p$.

\medskip

\noindent
{\bf 3.} \, 
Suppose that $\lambda$ with ${\rm Re}\; \lambda\geq0$ is an eigenvalue of $-L$.
This means
\begin{equation}
\begin{aligned}
\label{evp-u}
\lambda \rho u-\mu_* \Delta u +\nabla\pi &=0  &&\mbox{in }\;\Omega\setminus\Sigma,\\
{\rm div}\, u &=0 && \mbox{in }\;\Omega\setminus\Sigma,\\
P_\Sigma[\![u]\!] &=0  && \mbox{on } \Sigma,\\
-2P_\Sigma [\![\mu_*D(u)\nu_\Sigma]\!] - \theta_*{\sigma^\prime_*}\nabla_\Sigma \vartheta_\Sigma
 &=0 && \mbox{on } \Sigma,\\
 -2[\![\mu_*D(u) \nu_\Sigma\cdot\nu_\Sigma]\!] +\lj\pi\rj +\sigma_* \cA_\Sigma h
-\theta_*{\sigma^\prime_*}H_*\vartheta_\Sigma
 &=0 && \mbox{on } \Sigma, \\
u&=0 &&\mbox{on } \partial\Omega,
\end{aligned}
\end{equation}
\begin{equation}
\begin{aligned}
\label{evp-theta}
\lambda \rho\kappa_*\vartheta - d_*\Delta \vartheta &=0 && \mbox{in }\; \Omega\setminus\Sigma,\\
[\![\vartheta]\!]=0, \quad \vartheta &=\vartheta_\Sigma && \mbox{on } \Sigma,\\
\lambda \kappa_{\Gamma_*}\vartheta_\Sigma-d_{\Gamma_*}\Delta_\Sigma\vartheta_\Sigma
-(l_*/\theta_*)j_\Sigma-[\![d_*\partial_{\nu_\Sigma}\vartheta]\!] - {\sigma'_*}\text{div}_\Sigma u_\Sigma
&= 0 &&\mbox{on } \Sigma,\\
\partial_\nu\vartheta &=0 && \mbox{on } \partial\Omega,
\end{aligned}
\end{equation}
\begin{equation}
\hspace{1.4cm}
\begin{aligned}
\label{evp-h}
-2[\![\mu_*D(u)\nu_\Sigma\cdot\nu_\Sigma/\rho]\!] +\lj \pi/\rho\rj
+ l_*\vartheta_\Sigma + \gamma_* j_\Sigma &= 0 &&\mbox{on } \Sigma,\\
\lambda[\![\rho]\!]h -[\![\rho u\cdot\nu_\Sigma]\!] &= 0 &&\mbox{on }\Sigma.
\end{aligned}
\end{equation}
Observe that on $\Sigma$ we may write 
$$u_k = P_\Sigma u + \lambda h\nu_\Sigma + j_\Sigma\nu_\Sigma/\rho_k
=u_\Sigma+j_\Sigma\nu_\Sigma/\rho_k, \quad k=1,2.$$
By this identity, taking the inner product of the problem for $u$ with $u$ and integrating by parts we get
\begin{align*}
0&=\lambda |\rho^{1/2}u|_2^2 -({\rm div}\; T(u,\pi,\theta_*) |u)_2\\
&=\lambda |\rho^{1/2}u|_2^2 + 2|\mu_*^{1/2}D(u)|_2^2 \\
&\phantom{=}
+([\![T(u,\pi,\theta_*) \nu_\Sigma]\!]|P_\Sigma u + \lambda h\nu_\Sigma)_\Sigma
+ ([\![T(u,\pi,\theta_*) \nu_\Sigma\cdot\nu_\Sigma/\rho]\!]|j_\Sigma)_\Sigma\\
&=\lambda |\rho^{1/2}u|_2^2 + 2|\mu_*^{1/2}D(u)|_2^2
+\sigma_*\bar{\lambda} (\cA_\Sigma h|h)_\Sigma + l_*(\vartheta|j_\Sigma)_\Sigma\\
&\phantom{=}+\gamma_*|j_\Sigma|_\Sigma^2-\theta_*{\sigma^\prime_*}H_*\bar{\lambda}(\vartheta| h)_\Sigma
- \theta_*{\sigma^\prime_*}(\nabla_\Sigma\vartheta|P_\Sigma u)_\Sigma,
\end{align*}
since $[\![T(u,\pi,\theta_*)\nu_\Sigma]\!]=\sigma_* \cA_\Sigma h\nu_\Sigma
-\theta_*{\sigma^\prime_*}\vartheta_\Sigma H_*\nu_\Sigma- \theta_*{\sigma^\prime_*}\nabla_\Sigma \vartheta_\Sigma$ and, moreover,
$[\![T(u,\pi,\theta_*)\nu_\Sigma\cdot\nu_\Sigma/\rho]\!]= l_*\vartheta+\gamma_*j_\Sigma $.
On the other hand, the inner product of the equation for $\vartheta$ with  $\vartheta$
by an integration by parts and $[\![\vartheta]\!]=0$ leads to
\begin{align*}
0&= \lambda|(\rho\kappa_*)^{1/2}\vartheta|_2^2 + |d_*^{1/2}\nabla\vartheta|_2^2
+([\![d_*\partial_{\nu_\Sigma}\vartheta]\!]|\vartheta)_\Sigma\\
&= \lambda(|(\rho\kappa_*)^{1/2}\vartheta|_2^2 + |\kappa_{\Gamma_*}^{1/2}\vartheta_\Sigma|_\Sigma^2)
+ |d_*^{1/2}\nabla\vartheta|_2^2 + |d_{\Gamma_*}^{1/2}\nabla_\Sigma\vartheta_\Sigma|_{\Sigma}^2\\
& \phantom{=} - l_*(j_\Sigma|\vartheta)_\Sigma/\theta_*
+ {\sigma^\prime_*}(P_\Sigma u|\nabla_\Sigma\vartheta_\Sigma)_\Sigma + \lambda {\sigma^\prime_*}H_*(h|\vartheta)
\end{align*}
where we employed  $[\![d_*\partial_{\nu_\Sigma}\vartheta]\!]
= \kappa_{\Gamma_*}\lambda\vartheta_\Sigma-d_{\Gamma_*}\Delta_\Sigma\vartheta_\Sigma
-(l_*/\theta_*)j_\Sigma - {\sigma^\prime_*}\text{div}_\Sigma u_\Sigma$ and $u_\Sigma=P_\Sigma u +(u_\Sigma\cdot\nu_\Sigma)\nu_\Sigma$.
Adding the first identity to the second multiplied by $\theta_*$
and taking real parts yields the important relation
\begin{align}\label{evid}
0&={\rm Re}\,\lambda |\rho^{1/2}u|_2^2 + 2|\mu_*^{1/2}D(u)|_2^2 +\sigma_*{\rm Re}\,\lambda (\cA_\Sigma h|h)_\Sigma \nonumber\\
&+\theta_*({\rm Re}\,\lambda|(\rho\kappa_*)^{1/2}\vartheta|_2^2 + |d_*^{1/2}\nabla\vartheta|_2^2)\\
&+\gamma_*|j_\Sigma|_\Sigma^2 + \theta_*({\rm Re}\,\lambda |\kappa_{\Gamma_*}^{1/2}\vartheta_\Sigma|_\Sigma^2
+|d_{\Gamma_*}^{1/2}\nabla_\Sigma\vartheta_\Sigma|_\Sigma^2)\nn.
\end{align}
On the other hand, if ${\rm Im}\, \lambda\neq0$, taking imaginary parts separately we get
\begin{align*}
0&= {\rm Im}\, \lambda|\rho^{1/2}u|_2^2-{\rm Im}\, \lambda\, \sigma_* (\cA_\Sigma h|h)_\Sigma
+ {\rm Im}\,l_*(\vartheta|j_\Sigma)_\Sigma \\
&-  {\rm Im}\,\{\bar\lambda \theta_*{\sigma^\prime_*}H_*(\vartheta | h)_\Sigma\}
-  {\rm Im}\,\theta_* {\sigma^\prime_*} (\nabla_\Sigma\vartheta | P_\Sigma u)_\Sigma\\
0&= \theta_* {\rm Im}\, \lambda ( |(\rho\kappa_*)^{1/2}\vartheta|_2^2 + |\kappa_{\Gamma_*}^{1/2}\vartheta_\Sigma|_\Sigma^2)
 - {\rm Im}\, l_*(j_\Sigma |\vartheta)_\Sigma \\
 & +  {\rm Im}\, \{\lambda \theta_*{\sigma^\prime_*}H_*(h|\vartheta)_\Sigma\}
 + {\rm Im}\, \theta_* {\sigma^\prime_*} (P_\Sigma u| \nabla_\Sigma\vartheta)_\Sigma,
\end{align*}
hence
$$
\sigma_* (\cA_\Sigma h|h)_\Sigma = |\rho^{1/2}u|_2^2-\theta_*( |(\rho\kappa_*)^{1/2}\vartheta|_2^2
+|\kappa_{\Gamma_*}^{1/2}\vartheta_\Sigma|_\Sigma^2).
$$
Inserting this identity into \eqref{evid} leads to
$$
0=2{\rm Re}\,\lambda |\rho^{1/2}u|_2^2 + 2|\mu_*^{1/2}D(u)|_2^2
+ \theta_*|d_*^{1/2}\nabla\vartheta|_2^2 + \gamma_*|j_\Sigma|_\Sigma^2
+\theta_* |d_{\Gamma_*}^{1/2}\nabla_\Sigma\vartheta_\Sigma|_\Sigma^2 .
$$
This shows that if  $\lambda$ is an eigenvalue of $-L$ with ${\rm Re}\, \lambda\geq 0$ then $\lambda$ is real.
In fact, otherwise this identity implies $\vartheta=const=\vartheta_\Sigma$, $D(u)=0$ and $j_\Sigma=0$,
and then $u=0$ by Korn's inequality and the no-slip condition on $\partial\Omega$,
as well as $(\vartheta,\vartheta_\Sigma,h)=(0,0,0)$ by the equations for $\vartheta$ and $h$, since $\lambda\neq0$.

\medskip

\noindent
{\bf 4.} \, Suppose now that $\lambda>0$ is an eigenvalue of $-L$. Then we further have
$$
\lambda \int_{\Sigma} h d\Sigma = \int_\Sigma (u_k\cdot\nu_\Sigma -j_\Sigma/\rho_k) d\Sigma
= -\rho_k^{-1}\int_\Sigma j_\Sigma\, d\Sigma=0,
$$
as $\int_\Sigma u_k \cdot \nu_\Sigma\, d\Sigma= \int_{\Omega_k}{\rm div}\,u_k\,dx$.
Hence the mean values of $h$ and  $j_\Sigma$ both vanish
since the densities are non-equal.
Integrating the equations for $\vartheta$ and $\vartheta_\Sigma$, we obtain from this the relation
$$
\kappa_{\Gamma_*}\int_\Sigma\vartheta_\Sigma\,d\Sigma + \int_\Omega\rho \kappa_*\vartheta\,dx=0.
$$
Since  $\cA_\Sigma$ is positive semidefinite on functions with mean zero in case $\Sigma$ is connected,
by \eqref{evid} we obtain $(u,\vartheta,h)=(0,0,0)$, i.e.\ in this case there are no positive eigenvalues.
On the other hand, if $\Sigma$ is disconnected, there is at least one positive eigenvalue. To prove this we need some preparations.

\medskip

\noindent
{\bf 5.} \, First we consider the heat problem
\begin{equation}
\begin{aligned}
\label{NDdiffusion}
\lambda\rho\kappa_* \vartheta -d_*\Delta \vartheta &=0 &&\text{in }\; \Omega\setminus\Sigma,\\
[\![\vartheta]\!] &=0 &&\mbox{on } \Sigma,\\
\vartheta_\Sigma =\vartheta|_\Sigma &=g &&\mbox{on }\: \Sigma,\\
\partial_\nu\vartheta &=0 &&\mbox{on }\partial\Omega,
\end{aligned}
\end{equation}
and define $D_\lambda^Hg=-\lj d_*\pd_{\nu_\Sigma}\vartheta\rj$ on $\Sigma$, where
$D_\lambda^H$ denotes the Dirichlet-to-Neumann operator for this heat problem.
The properties of $D_\lambda^H$ are stated in \cite{PSW11}.
Then the solution $\vartheta$ of \eqref{evp-theta} can be expressed by
\begin{equation}
(\kappa_{\Gamma_*}\lambda - d_{\Gamma_*}\Delta_\Sigma + D_\lambda^H)\vartheta_\Sigma
-(l_*/\theta_*)j_\Sigma + \lambda {\sigma^\prime_*}H_*h - {\sigma^\prime_*}{\rm{div}}_\Sigma P_\Sigma u=0
\label{sol-heat}
\end{equation}
where we made use of the identity
$$
{\rm{div}}_\Sigma u_\Sigma={\rm{div}}_\Sigma P_\Sigma u - H_*u_\Sigma\cdot\nu_\Sigma={\rm{div}}_\Sigma P_\Sigma u-H_*\lambda h.
$$
\medskip
\goodbreak
\noindent
{\bf 6.} \, Next we solve the asymmetric Stokes problem
\begin{equation}
\begin{aligned}\label{asStokes}
\lambda \rho u-\mu_* \Delta u +\nabla\pi &=0 &&\mbox{in }\; \Omega\setminus\Sigma,\\
{\rm div}\, u &=0  &&\mbox{in }\;\Omega\setminus\Sigma,\\
P_\Sigma[\![u]\!] &= 0  &&\mbox{on }\Sigma,\\
 -[\![T(u,\pi,\theta_*) \nu_\Sigma\cdot\nu_\Sigma]\!] &= g_1 &&\mbox{on }\Sigma,\\
  -[\![T(u,\pi,\theta_*) \nu_\Sigma\cdot\nu_\Sigma/\rho]\!] &= g_2 &&\mbox{on } \Sigma,\\
  P_\Sigma[\![T(u,\pi,\theta_*) \nu_\Sigma]\!]&=g_3 &&\mbox{on }\Sigma,\\
  u&=0 &&\mbox{on } \partial\Omega,
\end{aligned}
\end{equation}
to obtain the output 
\begin{align*}
[\![\rho u\cdot\nu_\Sigma]\!]/[\![\rho]\!] &= S_\lambda^{11}g_1+S_\lambda^{12}g_2+S_\lambda^{13}g_3,\\
[\![u\cdot\nu_\Sigma]\!]/[\![1/\rho]\!] &= S_\lambda^{21}g_1+S_\lambda^{22}g_2+S_\lambda^{23}g_3,\\
P_\Sigma u&=S_\lambda^{31}g_1+S_\lambda^{32}g_2+S_\lambda^{33}g_3.
\end{align*}
Note that $g_1$, $g_2$, $(S_\lambda g)_1$, $(S_\lambda g)_2\in L_2(\Sigma)$ are scalar functions, while
$g_3$ and $(S_\lambda g)_3$ are vectors tangent to $\Sigma$, i.e., $g_3$ and $(S_\lambda g)_3\in L_2(\Sigma;T\Sigma)$,
with $T\Sigma$ being the tangent bundle of $\Sigma$.
For this problem we have
\begin{prop}\label{as-Stokes}
The operator $S_\lambda$ for the Stokes problem \eqref{asStokes} admits a bounded extension to
$L_{2}(\Sigma)^2\times L_2(\Sigma;T\Sigma)$ for $\lambda\geq0$ and has the following properties.\\
{\bf (i)} \, If $u$ denotes the solution of \eqref{asStokes}, then
$$ (S_\lambda g|g)_{L_2} \!=\! \lambda \int_\Omega \rho|u|^2\,dx
+ 2\int_\Omega \mu_*|D(u)|_2^2\, dx,\;\lambda\geq0, \; g\in L_{2}(\Sigma)^2\times L_2(\Sigma;T\Sigma) .$$
{\bf (ii)} \, $S_\lambda\in \cB(L_{2}(\Sigma)^2\times L_2(\Sigma;T\Sigma))$ is self-adjoint, positive semidefinite, and compact; in particular
\begin{align*}
& S^{11}_\lambda = [S_\lambda^{11}]^*,\quad S^{22}_\lambda = [S_\lambda^{22}]^*,\quad S^{33}_\lambda = [S_\lambda^{33}]^*\\
& S^{12}_\lambda = [S_\lambda^{21}]^*,\quad S^{13}_\lambda = [S_\lambda^{31}]^*,\quad S^{23}_\lambda = [S_\lambda^{32}]^*.
\end{align*}
{\bf (iii)} \, For each $\beta\in(0,1/2)$ there is a constant $C_\beta>0$ such that
$$ |S_\lambda|_{\cB(L_{2})}\leq \frac{C_\beta}{(1+\lambda)^\beta}, \quad \lambda\geq0.$$
{\bf (iv)}\, $|S_\lambda|_{\cB(L_{2},H^{1}_2)} \leq C$ uniformly for $\lambda\geq0$.
\vspace{2mm}\\
{\bf (v)} \, $S_\lambda^{11}, S_\lambda^{22}:
L_{2,0}(\Sigma)\to H^{1}_2(\Sigma)\cap L_{2,0}(\Sigma)$ are isomorphisms, for each $\lambda\geq0$,
where $L_{2,0}(\Sigma)=\{u\in L_2(\Sigma)\mid \int_\Sigma u\,d\Sigma=0\}$.
\end{prop}
\begin{proof} The assertions follow from similar arguments as in the proof of \cite[Proposition 4.3]{PSW13}.
\end{proof}
\medskip
\noindent
{\bf 7.} \,
The following lemma is needed in the proof of the main result of this section.
\begin{lem}\label{pos-def}
Let $H$, $V$ be Hilbert spaces.
Let $B$ be a positive definite operator on $H$,
$A:\cD(A)\subset H \to V$ be a closed, densely defined operator such that $A\cD(B^{1/2})\subset \cD(A^*)$.
Then $A^*A+B$ is a self-adjoint positive definite operator with
$\cD(A^*A+B)=\cD(B)$ and
$$
|A(A^*A+B)^{-1}A^*|_{\cB(V)}\le 1.
$$
In addition, if $A(A^*A+B)^{-1}A^*v=v$, then $v=0$.
\end{lem}
\begin{proof}
By the closed graph theorem $A:\cD(B^{1/2})\to V$ is bounded.
The closedness of $A^*$ in turn implies that the operator $A^*A:\cD(B^{1/2})\to H$ is closed.
Another application of the closed graph theorem then shows
that $A^*A: \cD(B^{1/2})\to H$ is bounded as well. This implies that $A^*A$ is a lower order perturbation of $B$ and also that
$A^*A+B$ is self-adjoint and positive definite. Therefore,
$(A^*A+B)^{-1}: H\to \cD(B)$ exists and is bounded.
\par
For $v\in \cD(A^*)$ we have with $K:=ACA^*:=A(A^*A+B)^{-1}A^*$
\begin{align*}
|Kv|_V^2
&= (CA^*v\mid A^*ACA^*v)_H\\
&= (CA^*v\mid A^*v)_H - (CA^*v\mid BCA^*v)_H\\\
&= (ACA^*v\mid v)_V - (BCA^*v\mid CA^*v)_H\\\
&=(Kv\mid v)_V - (Bw\mid w)_H
\end{align*}
with $w=CA^*v=(A^*A+B)^{-1}A^*v$. Since $B$ is positive definite, there exists $\beta>0$ such that
\begin{equation}
|Kv|_V^2 \le |Kv|_V|v|_V -\beta |w|_H^2 \le |Kv|_V|v|_V
\label{kv}
\end{equation}
 which shows $|K|_{\cB(V)}\le 1$. Moreover if $Kv=v$, then
 $|v|^2\le|v|^2-\beta|w|^2$ holds from \eqref{kv}. Hence $w=0$,
and consequently $v=Kv=Aw=0$.
\end{proof}
\noindent
Now suppose that $\lambda>0$ is an eigenvalue of $-L$. We set
\begin{equation}
g=
\begin{pmatrix}
\theta_*{\sigma^\prime_*}H_*\vartheta_\Sigma- \sigma_*\cA_\Sigma h\\
-l_* \vartheta_\Sigma-\gamma_*j_\Sigma\\
\theta_*{\sigma^\prime_*}\nabla_\Sigma\vartheta_\Sigma
\end{pmatrix}
=
\begin{pmatrix}
-\sigma_*\cA_\Sigma h\\
-\gamma_* j_\Sigma\\
0
\end{pmatrix}
+\theta_* Q\vartheta_\Sigma
\label{g}
\end{equation}
with $Q=({\sigma^\prime_*}H_*, -(l_*/\theta_*), {\sigma^\prime_*}\nabla_\Sigma)^{\sf T}$ to obtain
\begin{equation}
(\lambda h, j_\Sigma, P_\Sigma u)^{\sf T}
=S_\lambda g,\quad
S_\lambda=(S_\lambda^{ij})_{1\le i,j\le 3}.
\label{Slambda}
\end{equation}
We recall \eqref{sol-heat}. Since
\begin{align*}
-(l_*/\theta_*)j_\Sigma + \lambda {\sigma^\prime_*}H_*h \! - \!{\sigma^\prime_*}{\rm{div}}_\Sigma P_\Sigma u
&=Q^*S_\lambda g \\
&= \theta_*Q^*S_\lambda Q\vartheta_\Sigma \!-\! Q^*S_\lambda (\sigma_*\cA_\Sigma h, \gamma_* j_\Sigma, 0)^{\sf T},
\end{align*}
\eqref{sol-heat} is equivalent to
\begin{equation}
(\kappa_{\Gamma_*}\lambda - d_{\Gamma_*}\Delta_\Sigma + D_\lambda^H + \theta_*Q^*S_\lambda Q)\vartheta_\Sigma
= Q^*S_\lambda (\sigma_*\cA_\Sigma h, \gamma_* j_\Sigma, 0)^{\sf T}.
\label{vartheta-Gamma}
\end{equation}
Observing that $\kappa_{\Gamma_*}\lambda - d_{\Gamma_*}\Delta_\Sigma + D_\lambda^H + \theta_*Q^*S_\lambda Q$
is injective for $\lambda\ge 0$,
we solve the equation above for $\vartheta_\Sigma$ to the result
\begin{equation*}
\vartheta_\Sigma = L_\lambda Q^* S_\lambda ( \sigma_*\cA_\Sigma h, \gamma_* j_\Sigma, 0)^{\sf T}
\end{equation*}
with
$L_\lambda= (\kappa_{\Gamma_*}\lambda - d_{\Gamma_*}\Delta_\Sigma + D_\lambda^H + \theta_*Q^*S_\lambda Q)^{-1}$.
We set
$$
S_\lambda Q=(u_1,u_2,u_3)^{\sf T}.
$$
Combining \eqref{Slambda}, \eqref{g} and \eqref{vartheta-Gamma}, we obtain
\begin{align}
\begin{pmatrix}
\lambda h\\
j_\Sigma\\
P_\Sigma u
\end{pmatrix}
&=\theta_*S_\lambda Q L_\lambda (S_\lambda Q)^*
\begin{pmatrix}
\sigma_*\cA_\Sigma h\\
\gamma_* j_\Sigma\\
0
\end{pmatrix}
-S_\lambda
\begin{pmatrix}
\sigma_*\cA_\Sigma h\\
\gamma_* j_\Sigma\\
0
\end{pmatrix}\nn\\
&=-(S_\lambda -\theta_*S_\lambda Q L_\lambda Q^*S_\lambda )
\begin{pmatrix}
\sigma_*\cA_\Sigma h\\
\gamma_* j_\Sigma\\
0
\end{pmatrix}.
\label{simul}
\end{align}
In order to obtain the positivity of $S_\lambda -\theta_*S_\lambda Q L_\lambda Q^*S_\lambda$,
we symmetrize it as
$$
S_\lambda -\theta_*S_\lambda Q L_\lambda Q^*S_\lambda
= S_\lambda^{1/2} ( I - \theta_*S_\lambda^{1/2} Q L_\lambda Q^*S_\lambda^{1/2}) S_\lambda^{1/2}
=:S_\lambda^{1/2} (I-K) S_\lambda^{1/2},
$$
with 
$$K=\theta_*S_\lambda^{1/2} Q L_\lambda Q^*S_\lambda^{1/2}=A(A^*A+B)^{-1}A^*.$$ 
Here we have set $A=\theta_*^{1/2}S_\lambda^{1/2}Q$ with domain $\cD(A)=H_2^1(\Sigma)$ and 
$B=\kappa_{\Gamma_*}\lambda-d_{\Gamma_*}\Delta_\Sigma+D_\lambda^H$ 
with domain $\cD(B)=H_2^2(\Sigma)$. Furthermore, $H=L_2(\Sigma)$ and $V=L_{2}(\Sigma)^2\times L_2(\Sigma;T\Sigma)$.

By Lemma \ref{pos-def}, we know $|K|\le1$, therefore it holds that
\begin{align*}
(S_\lambda v -\theta_*S_\lambda Q L_\lambda Q^*S_\lambda v \mid v)
&= |S_\lambda^{1/2} v|^2 - (K S_\lambda^{1/2} v\mid S_\lambda^{1/2} v)\\
&\ge |S_\lambda^{1/2} v|^2 -|K| |S_\lambda^{1/2} v|^2\ge 0,
\end{align*}
which shows the positivity of $S_\lambda -\theta_*S_\lambda Q L_\lambda Q^*S_\lambda$.
\par
Writing the upper left $2\times 2$ block of $S_\lambda -\theta_*S_\lambda Q L_\lambda Q^*S_\lambda$ as
\begin{equation*}
S_\lambda^0 : =
\begin{pmatrix}
R_\lambda^1 & R_\lambda^*\\
R_\lambda & R_\lambda^2
\end{pmatrix},
\end{equation*}
$S_\lambda^0$ is also positive. Then by \eqref{simul} it holds that
\begin{equation}
\begin{pmatrix}
\lambda h\\
0
\end{pmatrix}
+S_\lambda^0
\begin{pmatrix}
\sigma_*\cA_\Sigma h \\
\gamma_* j_\Sigma
\end{pmatrix}
+\begin{pmatrix}
0\\
j_\Sigma
\end{pmatrix}
\!=\!
\begin{pmatrix}
\lambda h\\
0
\end{pmatrix}
+
\begin{pmatrix}
R_\lambda^1 & R_\lambda^*\\
R_\lambda   & R_\lambda^2+1/\gamma_*
\end{pmatrix}
\begin{pmatrix}
\sigma_*\cA_\Sigma h \\
\gamma_* j_\Sigma
\end{pmatrix}
\!=\!
\begin{pmatrix}
0\\
0
\end{pmatrix}.
\label{matrix}
\end{equation}
The following lemma is needed to solve \eqref{matrix}.

\begin{lem}[Schur]\label{Schur}
Let $H$ be a Hilbert space, $S,T,R\in\cB(H)$, $S=S^*$, $T=T^*$ and suppose that $T$ is invertible.
If
\begin{equation*}
\begin{pmatrix}
S & R^*\\
R & T
\end{pmatrix}\ge 0 \quad \text{on}\ H\times H,
\end{equation*}
then $S-R^*T^{-1}R\ge 0$ on $H$.
\end{lem}
\begin{proof}
For $x$ fixed, we set
$(x,y)^{\sf T} = (x, -T^{-1}Rx)^{\sf T}$. Then
\begin{align*}
0
&\le
\left(
\begin{pmatrix}
S & R^*\\
R & T
\end{pmatrix}
\begin{pmatrix}
x\\
y
\end{pmatrix}
\left|
\begin{pmatrix}
x\\
y
\end{pmatrix}
\right.\right)
=((S-R^*T^{-1}R)x\mid x),
\end{align*}
and this proves the assertion.
\end{proof}
Since $S^0_\lambda$ is positive, $R^1_\lambda$ and $R^2_\lambda$ are positive as well.
This implies, in particular, that $(R^2_\lambda +\varepsilon)$ is positive definite for any 
$\varepsilon>0$. Thus \eqref{matrix} is equivalent to the equation
\begin{equation}
\label{new}
\lambda h + (R^1_\lambda -R^*_\lambda(R^2_\lambda+\frac{1}{\gamma_*})^{-1}R_\lambda)\sigma_* \cA_\Sigma h=0.
\end{equation}
We first show that
\begin{equation*}
R^1_\lambda -R^*_\lambda(R^2_\lambda+\frac{1}{\gamma_*})^{-1}R_\lambda: L_{2,0}(\Sigma)\to L_{2,0}(\Sigma)
\end{equation*}
is injective.
Let
\begin{equation*}
\label{matrix-notation}
\begin{pmatrix}
S & R^*\\
R & T
\end{pmatrix}
:=\begin{pmatrix}
R^1_\lambda  & R^*_\lambda\\
  R_\lambda  & R^2_\lambda +\frac{1}{2\gamma_*}
\end{pmatrix}
\end{equation*}
and observe that the assertion of
 Lemma~\ref{Schur} holds true for this matrix with $H=L_{2,0}(\Sigma)$.
Suppose
$(R^1_\lambda -R^*_\lambda(R^2_\lambda+\frac{1}{\gamma_*})^{-1}R_\lambda)h=0$ for some $h\in L_{2,0}(\Sigma)$.
Then
\begin{equation*}
\begin{aligned}
0&=((R^1_\lambda -R^*_\lambda(R^2_\lambda+\frac{1}{\gamma_*})^{-1}R_\lambda)h|h)\\
 &=((S-R^*T^{-1}R)h|h) 
+\frac{1}{2\gamma_*}( T^{-1} (T+\frac{1}{2\gamma_*})^{-1}Rh|Rh)\ge c|Rh|^2\ge 0.
\end{aligned}
\end{equation*}
Thus, $ Rh=0$, and then also $Sh=0$.
(That is, $R^1_\lambda h=0$).
This implies
\begin{equation}
\label{scalar-prod}
0= (S_\lambda^0(h,0)|(h,0))= ((I-K)S_\lambda^{1/2}(h,0,0)|\S_\lambda^{1/2}(h,0,0)).
\end{equation}
We conclude that
$(I-K)S_\lambda^{1/2} (h,0,0)=(0,0,0)$, and Lemma~\ref{pos-def} then yields 
$S_\lambda^{1/2}(h,0,0)=(0,0,0)$. Therefore, $S_\lambda (h,0,0)=(0,0,0)$, and in particular $S_\lambda^{11} h =0$.
We can now, at last, infer from Proposition~\ref{as-Stokes}(v) that $h=0$. 

\smallskip
An analogous argument shows that $R^1_\lambda: L_{2,0}(\Sigma)\to L_{2,0}(\Sigma)$ is injective as well.
Indeed, if $R^1_\lambda h=0$ for some $h\in L_{2,0}(\Sigma)$, then  \eqref{scalar-prod} holds, and the proof proceeds 
just as above.
We note that $R^1_\lambda$ admits a representation
$$R^1_\lambda=S^{11}_\lambda(I-C_\lambda)$$ 
on $L_{2,0}(\Sigma)$,
with $C_\lambda$ a compact operator.
Since $R^1_\lambda$ 
is injective on $L_{2,0}(\Sigma)$, $(I-C_\lambda)$ must be so as well.
Consequently, $(I-C_\lambda)$ is a bijection as it has Fredholm index zero.
Proposition~\ref{as-Stokes}(v) then implies
$$
R^1_\lambda \in{\rm Isom}(L_{2,0}(\Sigma),L_{2,0}(\Sigma)\cap H^1_2(\Sigma)),
$$ 
i.e., $R^1_\lambda$ is an isomorphism
between the indicated spaces.

\smallskip\noindent
A similar argument now shows that
\begin{equation*}
R^1_\lambda -R^*_\lambda(R^2_\lambda+\frac{1}{\gamma_*})^{-1}R_\lambda
\in{\rm Isom}(L_{2,0}(\Sigma),L_{2,0}(\Sigma)\cap H^1_2(\Sigma)).
\end{equation*}
Setting $T_\lambda:=[R^1_\lambda -R^*_\lambda(R^2_\lambda+\frac{1}{\gamma_*})^{-1}R_\lambda]^{-1}$,
equation \eqref{new} can be written as
$$
\lambda T_\lambda h + \sigma_* \cA_\Sigma h=0.
$$
This equation can be treated in the same way as in \cite{PSW13}. As a conclusion,
$$
B_\lambda:=\lambda T_\lambda + \sigma_* \cA_\Sigma 
$$ 
has a nontrivial kernel for some $\lambda_0>0$,
which implies that $-L$ has a positive eigenvalue. Even more is true.
We have seen that $B_\lambda$ is positive definite for large $\lambda$ and
$B_0=\sigma_* \cA_\Sigma$ has $-\sigma (n-1)/R_*^2$ as an eigenvalue of multiplicity $m-1$ in $L_{2,0}(\Sigma)$.
Therefore, as $\lambda$ increases to infinity, $m-1$ eigenvalues $\mu_k(\lambda)$ of $B_\lambda$
must cross through zero, this way inducing $m-1$ positive eigenvalues of $-L$.

\medskip

\noindent
{\bf 8.} \, Next we look at the eigenvalue $\lambda=0$. Then \eqref{evid} yields
$$
2|\mu_*^{1/2}D(u)|_2^2+\gamma_*|j_\Sigma|_\Sigma^2
+ \theta_*(|d_*^{1/2}\nabla\vartheta|_2^2+|d_{\Gamma_*}^{1/2}\nabla_\Sigma\vartheta_\Sigma|_\Sigma^2)=0,
$$
hence  $\vartheta$ is constant, $D(u)=0$ and $j_\Sigma=0$ by the flux condition for $\vartheta$.
This further implies that $[\![u]\!]=0$, and therefore Korn's inequality yields $\nabla u=0$
and then we have $u=0$ by the no-slip condition on $\partial \Omega$.
This implies further that the pressures are constant in the phases and 
$[\![\pi]\!]=-\sigma_* \cA_\Sigma h +\theta_*\sigma'_*H_*\vartheta_\Sigma,$
as well as $[\![\pi/\rho]\!]=-l_*\vartheta$.
Thus the dimension of the eigenspace for eigenvalue $\lambda=0$ is the same as the dimension of the manifold of equilibria,
namely $mn +2$ if $\Omega_1$ has $m\geq1$ components.
We set $\Sigma=\cup_{1\le k\le m}\Sigma_k$. Hence, $\Sigma$ consists of $m$ spheres $\Sigma_k$, $k=1,\ldots, m$, of equal radius with
$\overline{\Sigma_k}\cap \overline{\Sigma_j}=\emptyset$, $l\ne j$, $\overline{\Sigma_k}\subset \Omega$, $k=1,\ldots, m$.
The kernel of $L$ is spanned by $e_{\vartheta,\vartheta_\Sigma}=(0,1,1,0)$, $e_h=(0,0,0,1)$, $e_{ik}=(0,0,0,Y_k^i)$
with the spherical harmonics $Y_k^i$ of degree one for the spheres $\Sigma_k$, $i=1,\ldots,n$, $k=1,\ldots, m$.

To show that the equilibria are normally stable or normally hyperbolic,
it remains to prove that $\lambda=0$ is semi-simple.
So suppose we have a solution of $L(u,\vartheta,h)= \sum_{i,k}\alpha_{ik}e_{ik} + 
\beta e_{\theta,\theta_\Sigma} +\delta e_h$. This means
\begin{equation}
\hspace{-1mm}
\begin{aligned}
\label{ev-St}
-\mu_* \Delta u +\nabla\pi &=0 &&\mbox{in }\; \Omega\setminus\Sigma,\\
{\rm div}\, u &=0 &&\mbox{in }\;\Omega\setminus\Sigma,\\
P_\Sigma[\![u]\!] &=0 &&\mbox{on } \Sigma,\\
\!\!-[\![ T(u,\pi,\theta_*)\nu_\Sigma]\!]  + \sigma_* \cA_\Sigma h \nu_\Sigma
\!-\!\theta_*{\sigma^\prime_*}H_*\vartheta_\Sigma \nu_\Sigma\!-\!\theta_*{\sigma'_*}\nabla_\Sigma \vartheta_\Sigma &=0 &&
\mbox{in }\;\Sigma,\\
u&=0 &&\mbox{on } \partial\Omega.
\end{aligned}
\end{equation}
\\
\begin{equation}
\begin{aligned}\label{ev-Ht}
 -d_*\Delta \vartheta &=\rho\kappa_*\beta &&\mbox{in }\;\Omega\setminus\Sigma,\\
[\![\vartheta]\!]&=0 && \mbox{on }\Sigma,\\
- d_{\Gamma_*}\Delta_\Sigma\vartheta_\Sigma
-(l_*/\theta_*)j_\Sigma - [\![d_*\partial_{\nu_\Sigma}\vartheta]\!] - \sigma'_* \text{div}_\Sigma u_\Sigma
&= \kappa_{\Gamma_*}\beta &&\mbox{on }\Sigma,\\
\partial_\nu\vartheta &=0 &&\mbox{on } \partial\Omega.
\end{aligned}
\end{equation}
\\
\begin{equation}
\hspace{6mm}
\begin{aligned}\label{ev-j}
-[\![T(u,\pi,\theta_*)\nu_\Sigma\cdot\nu_\Sigma/\rho]\!]
+l_*\vartheta_\Sigma +\delta_* j_\Sigma &= 0 &&\mbox{on } \Sigma,\\
 -[\![\rho u\cdot \nu_\Sigma]\!]/[\![\rho]\!] &= \sum_{i,k} \alpha_{ik}Y_k^i +\delta &&\mbox{on } \Sigma.
\end{aligned}
\end{equation}
We have to show $\alpha_{ik}=\beta=\delta=0$ for all $i,k$.
By ${\rm div}\, u=0$, we have
$$
\delta|\Sigma|=\delta|\Sigma|+\sum_{i,k}\alpha_{ik}\int_\Sigma Y_k^i\,d\Sigma =
-\int _\Sigma [\![\rho u\cdot \nu_\Sigma]\!]/[\![\rho]\!]\,d\Sigma=0, 
$$
which implies $\delta=0$. Since $\int_\Sigma j_\Sigma\,d\Sigma =0$
we have
\begin{align*}
\beta[(\rho\kappa_*|1)_\Omega+\kappa_{\Gamma_*}|\Sigma|]
&=-{\sigma^\prime_*}\int_\Sigma{\rm div}_\Sigma\,u_\Sigma\,d\Sigma={\sigma^\prime_*}H_*\int_\Sigma u_\Sigma\cdot\nu_\Sigma\,d\Sigma\\
&={\sigma^\prime_*}H_*\int_\Sigma  [\![\rho u\cdot \nu_\Sigma]\!]/[\![\rho]\!]\,d\Sigma =0,
\end{align*}
which implies $\beta=0$.
\par
Integrating and adding up the first equation of \eqref{ev-St} multiplied by $u$, the fourth equation of \eqref{ev-St} multiplied by $u_\Sigma$,
the first equation of \eqref{ev-Ht} multiplied by $\theta_*\vartheta$, the third equation of \eqref{ev-Ht} multiplied by $\theta_*\vartheta_\Sigma$,
and the first equation of \eqref{ev-j} multipled by $j_\Sigma$, we obtain
$$
2|\mu_*^{1/2}D(u)|_2^2+\gamma_*|j_\Sigma|_\Sigma^2
+ \theta_*(|d_*^{1/2}\nabla\vartheta|_2^2+|d_{\Gamma_*}^{1/2}\nabla_\Sigma\vartheta_\Sigma|_\Sigma^2)
+\sigma_*(\cA_\Sigma h|\nu_\Sigma\cdot u_\Sigma)_\Sigma=0.
$$
Next we observe that
\begin{equation*}
\begin{aligned}
(\cA_\Sigma h\,|\,\nu_\Sigma\cdot u_\Sigma)_\Sigma
&=(\cA_\Sigma h \,|\, [\![\rho u\cdot \nu_\Sigma]\!]/[\![\rho]\!])_\Sigma 
=-\sum_{i,k} \alpha_{ik}(\cA_\Sigma h \,|\, Y_k^i)_\Sigma=0,
\end{aligned}
\end{equation*}
since $\cA_\Sigma$ is self-adjoint and 
the spherical harmonics $Y_k^i$ are in the kernel of $\cA_\Sigma$.
We can now conclude that $u_\Sigma=u-j_\Sigma\nu_\Sigma/\rho=0$, and hence
$$
0=u_\Sigma\cdot\nu_\Sigma=-\sum_{i,k}\alpha_{ik}Y_k^i.
$$
Thus $\alpha_{ik}=0$ for $1\le i\le n$, $1\le k\le m$, as the spherical harmonics $Y_k^i$ are linearly independent.
Therefore, the eigenvalue $\lambda=0$ is semi-simple.

\medskip
\goodbreak

\noindent
{\bf 9.} \, Let us summarize what we have proved.

\begin{thm}\label{propertiesL} Let $L$ denote the linearization at  $e_*:=(0,\theta_*,\theta_*|_\Sigma,\Sigma)\in\cE$ as defined above.
Then $-L$ generates a compact analytic $C_0$-semigroup in $X_0$ which has maximal $L_p$-regularity. The spectrum of $L$ consists only of eigenvalues of finite algebraic multiplicity. Moreover, the following assertions are valid.
\begin{itemize}
\item[(i)] The operator $-L$ has no eigenvalues $\lambda\neq0$ with nonnegative real part if and only if $\Sigma$ is connected.
\vspace{2mm}
\item[(ii)] If $\Sigma$ is disconnected, then $-L$ has precisely $m-1$ positive eigenvalues.
\vspace{2mm}
\item[(iii)] $\lambda=0$ is an eigenvalue of $L$ and it is semi-simple.
\vspace{2mm}
\item[(iv)] The kernel $N(L)$ of $L$ is isomorphic to the tangent space $T_{e_*}\cE$ of the manifold of equilibria $\cE$ at $e_*$.
\end{itemize}
\noindent
Consequently, $e_*=(0,\theta_*,\theta_*|_\Sigma,\Sigma)\in\cE$ is normally stable if and only if  $\Sigma$ is connected,
and normally hyperbolic if and only if  $\Sigma$ is disconnected.
\end{thm}

\section{Nonlinear Stability of Equilibria}

\noindent
{\bf 1.} \, We look at Problem \eqref{bulk}, \eqref{interface} in the neighborhood of a non-degenerate equilibrium
$e_*=(0,\theta_*,\Gamma_*)\in\cE$. Employing a Hanzawa transform with reference manifold
$\Sigma=\Gamma_*$ as in \cite[Section 3]{PSW13}, the transformed problem becomes
\begin{equation}
\begin{aligned}\label{enonlin-u}
\rho\partial_t u-\mu_* \Delta u +\nabla\pi &=F_u(u,\pi,\vartheta,h),\\
{\rm div}\, u &=G_d(u,h) \\
P_\Sigma[\![u]\!] &= G_u(u,h)\\
-2P_\Sigma[\![\mu_*D(u)\nu_\Sigma]\!] -\theta_*\sigma'_*\nabla_\Sigma\vartheta_\Sigma
&= G_\tau(u,\vartheta,\vartheta_\Sigma,h),\\
 -2 [\![\mu_* D(u) \nu_\Sigma\cdot\nu_\Sigma]\!] +\lj\pi\rj
 +\sigma_* \cA_\Sigma h -\theta_*\sigma'_* H_* \vartheta_\Sigma
 &=G_\nu(u,\vartheta,h)+G_\gamma ,\\
-2 [\![\mu_* D(u) \nu_\Sigma\cdot\nu_\Sigma/\rho]\!] +\lj\pi/\rho\rj +l_*\vartheta_\Sigma +\gamma_*j_\Sigma
&= G_h(u,\vartheta,h),\\
u&=0,\\
u(0)&=u_0, 
\end{aligned}
\end{equation}
\\
with $G_\gamma=G_\gamma(\vartheta_\Sigma,h)$, where $\mu_*=\mu(\theta_*)$, $\sigma_*=\sigma(\theta_*)$, $\sigma'_*=\sigma'(\theta_*)$, $l_*=l(\theta_*)$, $\gamma_*=\gamma(\theta_*)$, and
$$\cA_\Sigma =-H^\prime(0)=-(n-1)/R_*^2 -\Delta_{\Sigma}.$$
For the relative temperature $\vartheta=(\theta-\theta_*)/\theta_*$ we obtain
\begin{equation}
\begin{aligned}\label{enonlin-theta}
\rho\kappa_*\partial_t\vartheta -d_*\Delta \vartheta &=F_\theta(u,\vartheta,h)
&&\mbox{in }\; \Omega\setminus{\Sigma},\\
\!\kappa_{\Gamma_*}\pd_t\vartheta_\Sigma -d_{\Gamma_*}\Delta_\Sigma\vartheta_\Sigma
-(l_*/\theta_*)j_\Sigma -\\
-[\![d_*\partial_{\nu_\Sigma}\vartheta]\!] - {\sigma'_*} \text{div}_\Sigma u_\Sigma 
&=G_\theta(u,\vartheta,\vartheta_\Sigma,h ) &&\mbox{on } {\Sigma},\\
[\![\vartheta]\!]&=0 && \mbox{on }{\Sigma},\\
\partial_\nu\vartheta&=0 && \mbox{on }\partial\Omega,\\
\vartheta(0)&=\vartheta_0 &&\mbox{in }\; \Omega,
\end{aligned}
\end{equation}
with $\kappa_*=\kappa(\theta_*)$, $d_*=d(\theta_*)$, $\kappa_{\Gamma*}=\kappa_\Gamma(\theta_*)$,
$d_{\Gamma*}=d_\Gamma(\theta_*)$. Finally, the evolution of $h$ is determined by
\begin{equation}
\begin{aligned}\label{enonlin-h}
\partial_t h -[\![\rho u\cdot\nu_\Sigma]\!]/[\![\rho]\!]&=F_h(u,h) &&\mbox{ on } {\Sigma},\\
h(0)&=h_0.&&
\end{aligned}
\end{equation}
Here
$F_u$, $G_d$, $G_u$, $G_\nu$, $F_\theta$, $F_h$ are the same as in the case $\sigma$ where is constant;
see \cite[Section 3]{PSW13},
and
\begin{align*}
G_\tau(u,\vartheta,\vartheta_\Sigma, h)  &=  2 P_\Sigma[\![(\mu(\theta)-\mu(\theta_*))D(u)\nu_\Sigma]\!]
- 2P_\Sigma [\![\mu(\theta)D(u)M_0(h)\nabla_\Sigma h]\!]\\
&-P_\Sigma [\![\mu(\theta)(M_1(h)\nabla u+[M_1(h)\nabla u]^{\sf T})
(\nu_\Sigma-M_0(h)\nabla_\Sigma h)]\!]\\
&+[\![\mu(\theta)((I-M_1)\nabla u+[(I-M_1)\nabla u]^T)\\
&\phantom{+[\![\mu(\theta)((I-M_1)\nabla u}
(\nu_\Sigma-M_0\nabla_\Sigma h)\cdot\nu_\Sigma]\!]M_0(h)\nabla_\Sigma h\\
& -\theta_*(\sigma^\prime_*-(\sigma^\prime(\theta_\Sigma)/\beta(h))\nabla_\Sigma \vartheta_\Sigma\\ 
& -\theta_*(\sigma^\prime(\theta_\Sigma)/\beta(h))P_\Sigma(I-P_\Gamma(h) M_0(h))\nabla_\Sigma \vartheta_\Sigma\\
&+(\sigma^\prime_*(\theta_\Sigma)/\beta(h))(P_\Gamma(h) M_0(h)\nabla_\Sigma\vartheta_\Sigma\cdot \nu_\Sigma)M_0(h)\nabla_\Sigma h,\\
\\
G_\gamma(\vartheta_\Sigma, h)
& =\sigma(\theta_\Sigma)H(h)-\sigma(\theta_*)H^\prime(0)h-\theta_*\sigma^\prime_* H_*\vartheta_\Sigma\\
& +\theta_*\sigma^\prime(\theta_\Sigma)/\beta(h)(P_\Gamma(h) M_0(h)\nabla_\Sigma\vartheta_\Sigma\cdot M_0(h)\nabla_\Sigma h),\\
\\
G_\theta(u,\vartheta,\vartheta_\Sigma, h) & =
 (\kappa_{\Gamma_*}-\kappa_\Gamma(\theta_\Sigma))\pd_t\vartheta_\Sigma
-\kappa_\Gamma(\theta_\Sigma)u_\Sigma\cdot P_\Gamma(h)M_0(h)\nabla_\Sigma\vartheta_\Sigma\\
& -(d_{\Gamma_*} - d_\Gamma(\theta_\Sigma))\Delta_\Sigma \vartheta_\Sigma\\
& -(d_{\Gamma}(\theta_\Sigma)\Delta_\Sigma \vartheta_\Sigma - {\rm tr}\{P_\Gamma(h)M_0(h)\nabla_\Sigma\\
&\hspace{4cm}
(d_\Gamma(\theta_\Sigma)P_\Gamma(h)M_0(h)\nabla_\Sigma \vartheta_\Sigma)\})\\
& - \lj d_*\pd \nu_\Sigma\vartheta\rj + (\lj d(\theta)(I-M_1(h))\nabla_\Sigma\vartheta_\Sigma\cdot \nu_\Gamma\rj\\
&-(\sigma^\prime_* {\rm div}\,u_\Sigma-
(\theta_\Sigma/\theta_*)\sigma^\prime(\theta_\Sigma){\rm tr}\{P_\Gamma(h) M_0(h)\nabla_\Sigma u_\Sigma\})\\
& -(1/\theta_*)(l(\theta_*) j_\Sigma -l(\theta_\Sigma )j_\Gamma)+(\gamma(\theta_\Sigma)/\theta_*)j_\Gamma^2, \\
\\
G_h(u,\vartheta,h) & =
2\lj (\mu(\theta) -\mu(\theta_*))D(u)\nu_\Sigma\cdot\nu_\Sigma/\rho\rj\\
& -2( \lj \mu(\theta) D(u)\nu_\Sigma\cdot\nu_\Sigma/\rho \rj 
-\lj \mu(\theta)D(u)\nu_\Gamma\cdot\nu_\Gamma /\rho\rj )\\
& -\lj \mu(\theta)(M_1(h) + [M_1(h)]^{\sf T})\nu_\Gamma\cdot\nu_\Gamma/\rho\rj \\  
& -(\lj \psi(\theta)\rj +\theta_* \lj \psi^\prime (\theta_*)\rj -\theta_*\lj \psi^\prime(\theta_*)\rj \vartheta_\Sigma)\\
& -(\gamma(\theta)j_\Gamma - \gamma(\theta_*)j_\Sigma) - \lj (1/2\rho^2)\rj j_\Gamma^2,
\end{align*}
cf. \cite[Section 2]{PSW13} for the definition of $M_0(h)$, $M_1(h)$, $P_\Gamma(h)$, and $\beta(h)$.\\
Moreover, 
$$\nu_\Gamma=\beta(h)(\nu_\Sigma-M_0(h)\nabla_\Sigma h),
\quad j_\Gamma= \lj u\cdot \nu_\Gamma\rj /\lj 1/\rho\rj,
\quad j_\Sigma= \lj u\cdot \nu_\Sigma\rj /\lj 1/\rho\rj.
$$
The nonlinearities are  $C^1$ from $\EE$ to $\FF$, satisfying
$(F_i^\prime(0),G_k^\prime(0))=(0,0)$ 
for all
$i\in\{u,\theta,h\}$ and $k\in\{d,u,\tau,\nu,\gamma,h,\theta\}$.
\\
\noindent
The state manifold locally near the equilibrium $e_*=(0,\theta_*,\Gamma_*)$  reads as
\begin{align}\label{phasemanifeq}
\cSM:=\Big\{& (u,\vartheta, h)\in L_p(\Omega)^{n+1}\times  C^2(\Sigma):\\
&(u,\vartheta)\in W^{2-2/p}_p(\Omega\setminus\Sigma)^{n+1},\
 h\in W^{3-2/p}_p(\Sigma),\nonumber\\
 &{\rm div}\, u=G_d(u,h) \mbox{ in } \Omega\setminus\Sigma,\quad u=0,\;\partial_\nu\theta =0 \mbox{ on } \partial\Omega,\nn\\
 & P_\Sigma[\![u]\!] = G_u(u,\vartheta,h),\;\lj \vartheta \rj =0\mbox{ on }\Sigma,\nonumber\\
 & -2 P_\Sigma[\![\mu_*D(u)\nu_\Sigma]\!] - \theta_*{\sigma^\prime_*}\nabla_\Sigma\vartheta
 =G_\tau(u,\vartheta,\vartheta_\Sigma, h) \mbox{ on }\Sigma\Big\}.\nonumber
\end{align}
Note that only $(u,\vartheta,h)$ need to be considered as state manifold for the flows, as $\vartheta_\Sigma$ is the
trace of $\vartheta$, and the pressure $\pi$ can be recovered from the flow variables, as in \cite{PSW13}.
Due to the compatibility conditions this is a nonlinear manifold.
By parameterizing this manifold over its tangent space
\begin{align*}
\cSX:=\Big\{ &(u,\vartheta, h)\in L_p(\Omega)^{n+1}\times C^2(\Sigma):\\
 &(u,\vartheta)\in W^{2-2/p}_p(\Omega\setminus\Sigma)^{n+1},\
 h\in W^{3-2/p}_p(\Sigma),\nonumber\\
 &{\rm div}\, u=0\mbox{ in }\Omega\setminus\Sigma,\quad u=0,\;\partial_\nu\theta =0 \mbox{ on } \partial\Omega,\\
 & P_\Sigma[\![u]\!] = 0,\; \lj \vartheta \rj =0,\; 
  -2 P_\Sigma[\![\mu_* D(u)\nu_\Sigma]\!] - \theta_*{\sigma^\prime_*}\nabla_\Sigma\vartheta
 =0 \mbox{ on }\Sigma\Big\}, \nonumber
\end{align*}
the nonlinear problem \eqref{enonlin-u}, \eqref{enonlin-theta}, \eqref{enonlin-h} is written in the form studied in \cite{PSW13},
where the generalized principle of linear stability is proved for the problem with constant surface tension.
An adaption of this proof implies the following result.

\begin{thm} \label{stability} Let $p>n+2$, $\rho_1,\rho_2>0$, $\rho_1\neq\rho_2$,  and  suppose
the assumptions {\rm{\bf a), b), c)}} in Section 1.
Then in the topology of the state manifold $\cSM$ we have:
\begin{itemize}
\item[(i)]  $(0,\theta_*,\Gamma_*)\in\cE$ is stable if and only if $\Gamma_*$ is connected.\\
\vspace{-2mm}
\item[(ii)]  Any solution starting in a neighborhood of a stable equilibrium exists globally and converges to
a possibly different stable equilibrium in the topology of $\cSM$.\\
\vspace{-2mm}
\item[(iii)] Any solution starting and {\em staying in} a neighborhood of an unstable equilibrium exists globally
and converges to a possibly different unstable equilibrium in the topology of $\cSM$.
\end{itemize}
\end{thm}

\section{Qualitative Behaviour of the Semiflow} 

In this section we study the global properties of problem \eqref{bulk}, \eqref{interface},
following the approach of \cite[Section 6]{PSW13}.

Recall that the closed $C^2$-hypersurfaces contained in $\Omega$ form a $C^2$-manifold,
which we denote by $\cMH^2(\Omega)$ as in \cite[Section 2]{PSW13}.
The charts are the parameterizations over a given hypersurface $\Sigma$, and the tangent
space consists of the normal vector fields on $\Sigma$.
We define a metric on $\cMH^2(\Omega)$ by means of
$$d_{\cMH^2}(\Sigma_1,\Sigma_2):= d_H(\cN^2\Sigma_1,\cN^2\Sigma_2),$$
where $d_H$ denotes the Hausdorff metric on the compact subsets of $\R^n$.
This way $\cMH^2(\Omega)$ becomes a Banach manifold of class $C^2$.

Let $d_\Sigma(x)$ denote the signed distance for $\Sigma$.
We may then define the {\em level function} $\varphi_\Sigma$ by means of
$$\varphi_\Sigma(x) = \phi(d_\Sigma(x)),\quad x\in\R^n,$$
where
$$\phi(s)=s(1-\chi(s/a))+{\rm sgn}\, s \chi(s/a),\quad s\in \R.$$
It is easy to see that $\Sigma=\varphi_\Sigma^{-1}(0)$, and $\nabla \varphi_\Sigma(x)=\nu_\Sigma(x)$, for each $x\in \Sigma$.
Moreover, $\kappa=0$ is an eigenvalue of $\nabla^2\varphi_\Sigma(x)$,
the remaining eigenvalues of $\nabla^2\varphi_\Sigma(x)$ are the principal curvatures $\kappa_j$ of $\Sigma$ at $x\in\Sigma$.

If we consider the subset $\cMH^2(\Omega,r)$ of $\cMH^2(\Omega)$ which consists of all closed hyper-surfaces
$\Gamma\in \cMH^2(\Omega)$ such that $\Gamma\subset \Omega$ satisfies the ball condition
with fixed radius $r>0$ then the map $\Phi:\cMH^2(\Omega,r)\to C^2(\bar{\Omega})$ defined by
$\Phi(\Gamma)=\varphi_\Gamma$ is an isomorphism of
the metric space $\cMH^2(\Omega,r)$ onto $\Phi(\cMH^2(\Omega,r))\subset C^2(\bar{\Omega})$.

Let $s-(n-1)/p>2$; for $\Gamma\in\cMH^2(\Omega,r)$, we define $\Gamma\in W^s_p(\Omega,r)$
if $\varphi_\Gamma\in W^s_p(\Omega)$. In this case the local charts for $\Gamma$ can be chosen of class $W^s_p$ as well.
A subset $A\subset W^s_p(\Omega,r)$ is  (relatively) compact, if and only if  $\Phi(A)\subset W^s_p(\Omega)$ is (relatively) compact.

As an ambient space for the
{\bf state manifold }$\cSM_\Gamma$ of Problem  (\ref{NS}), (\ref{Heat}), (\ref{GTS})  we consider
the product space $L_p({\Omega})^{n+1}\times \cMH^2(\Omega)$ and set

\begin{align}\label{phasemanigamma}
\cSM_\Gamma:=\Big\{ &(u,\theta, \Gamma)\in L_p(\Omega)^{n+1}\times \cMH^2(\Omega): \nonumber\\
& (u,\theta)\in W^{2-2/p}_p(\Omega\setminus\Gamma)^{n+1},\ \ 
0<\theta<\theta_c,\ \ \Gamma\in W^{3-2/p}_p,\nonumber\\
&{\rm div}\, u=0\,\mbox{ in } \Omega\setminus\Gamma,\quad u=\partial_\nu\theta =0 \,\mbox{ on } \partial\Omega,\nn\\
& P_\Gamma[\![u]\!] = 0,\; \lj \theta \rj =0
\, \mbox{ on }\ \Gamma,
\nonumber\\ & 
2 P_\Gamma\lj \mu(\theta) D(u)\nu_\Gamma \rj + \sigma'(\theta)\nabla_\Gamma\theta
=0 \mbox{ on }\Gamma\Big\}.\nonumber
\end{align}
\noindent
Charts for these manifolds are obtained by the charts induced by $\cMH^2(\Omega)$,
followed by a Hanzawa transformation.

Applying Theorem \ref{wellposed} and re-parameterizing the interface repeatedly,
we see that (\ref{NS}), (\ref{Heat}), (\ref{GTS}) yields a local semiflow on $\cSM_\Gamma$.

\begin{thm}\label{semiflow} Let $p>n+2$, $\sigma,\rho_1,\rho_2>0$, $\rho_1\neq\rho_2$, and suppose
the assumptions {\rm{\bf a), b), c)}} in Section 1. \par
Then problem \eqref{bulk}, \eqref{interface} generates a local semiflow
on the state manifold $\cSM_\Gamma$. Each solution $(u,\theta,\Gamma)$ of the problem exists on a maximal time
interval $[0,t_*)$, where $t_*=t_*(u_0,\theta_0,\Gamma_0)$.
\end{thm}

\noindent
Again we note that the pressure $\pi$ as well as the phase flux $j_\Gamma$ and $\theta_\Gamma$
are dummy variables which are determined for each $t$ by the principal variables
$(u,\theta,\Gamma)$. In fact,
$$
j_\Gamma= [\![u\cdot\nu_\Gamma]\!]/[\![1/\rho]\!],
$$
$\theta_\Gamma$ is the trace of $\theta$, and $\pi$ is determined by the weak transmission problem
\begin{align*}
&(\nabla\pi|\nabla\phi/\rho)_{L_2(\Omega)}\\
& =(2\rho^{-1}{\rm div}(\mu(\theta)D(u))-u\cdot\nabla u|\nabla\phi)_{L_2(\Omega)},
\quad \phi\in H^1_{p^\prime}(\Omega),\; \phi=0 \mbox{ on } \Gamma,\\
&[\![\pi]\!]= 2[\![\mu(\theta)D(u)\nu_\Gamma]\!]\nu_\Gamma+\sigma(\theta_\Gamma)H_\Gamma - \lj 1/\rho\rj j_\Gamma^2
\; \mbox{ on } \Gamma,\\
& [\![\pi/\rho]\!]
= 2[\![(\mu(\theta)/\rho) D(u)\nu_\Gamma\cdot\nu_\Gamma]\!]
- [\![1/(2\rho^2)]\!]j_\Gamma^2-[\![\psi(\theta)]\!]- \gamma(\theta_\Gamma)j_\Gamma \;\mbox{ on } \Gamma,
\end{align*}
Concerning such transmission problems we refer to \cite{KPW10}.

\medskip

\subsection{Convergence}

There are several obstructions against global existence:
\begin{itemize}
\item {\em regularity}: the norms of either $u(t)$, $\theta(t)$, or $\Gamma(t)$  become unbounded;
\vspace{2mm}
\item {\em geometry}: the topology of the interface changes;\\
    or the interface touches the boundary of $\Omega$.
\vspace{2mm}
\item {\em well-posedness}: the temperature leaves the range $0<\theta(t)<\theta_c$.
\end{itemize}
Note that  the compatibility conditions,
\begin{align*}
&{\rm div}\, u(t)=0\mbox{  in } \Omega\setminus\Gamma(t), \quad u=0,\;\partial_\nu\theta=0\mbox{ on }\partial\Omega,\\
& P_\Gamma [\![u(t)]\!]=0,\;[\![\theta]\!]=0,\;P_\Gamma[\![\mu(\theta)(\nabla u+[\nabla u]^{\sf T})\nu_\Gamma]\!]
+ \sigma'(\theta_\Gamma)\nabla_\Gamma\theta_\Gamma
=0 \mbox{ on }\Gamma(t),
 \end{align*}
are preserved by the semiflow.

Let $(u,\theta,\Gamma)$ be a solution in the state manifold $\cSM_\Gamma$ with maximal interval $[0,t_*)$. By the
{\em uniform ball condition} we mean the existence of a radius $r_0>0$ such that for each $t$,
at each point $x\in\Gamma(t)$ there exists centers $x_i\in \Omega_i(t)$ such that
$B_{r_0}(x_i)\subset \Omega_i$ and $\Gamma(t)\cap \bar{B}_{r_0}(x_i)=\{x\}$, $i=1,2$. Note that this condition
bounds the curvature of $\Gamma(t)$, prevents parts of it to  touch the outer
boundary $\partial \Omega$, and to undergo topological changes. Hence if this condition holds, then the volumes  of the phases are preserved.

With this property, combining the local semiflow for \eqref{bulk}, \eqref{interface} with
the Ljapunov functional, i.e., the negative total entropy, and compactness we obtain the following result.

\bigskip

\begin{thm} \label{Qual} Let $p>n+2$, $\rho_1,\rho_2>0$, $\rho_1\neq\rho_2$, and suppose
the assumptions {\rm{\bf a), b), c)}} in Section 1.
Suppose that $(u,\theta,\Gamma)$ is a solution of
\eqref{bulk}, \eqref{interface} in the state manifold $\cSM_\Gamma$ on its maximal time interval $[0,t_*)$.
Assume that the  following conditions hold on $[0,t_*)$:
\begin{itemize}
\item[(i)]  \, $\sup_{0<t<t_*}(|u(t)|_{{W^{2-2/p}_p}},\ |\theta(t)|_{W^{2-2/p}_p},
\ |\Gamma(t)|_{W^{3-2/p}_p})<\infty$; \\
\vspace{-2mm}
\item[(ii)] \, $\Gamma(t)$ satisfies the uniform ball condition.\\
\vspace{-2mm}
\item[(iii)] \, $\inf_{0<t<t_*}\theta(t)>0$ and $\sup_{0<t<t_*}\theta(t)<\theta_c$ on $\bar{\Omega}$.
\end{itemize}
\smallskip
Then $t_*=\infty$, i.e.\ the solution exists globally, and its limit set $\omega_+(u,\theta,\Gamma)\subset\cE$ is non-empty. If further $(0,\theta_\infty,\Gamma_\infty)\in\omega_+(u,\theta,\Gamma)$ with $\Gamma_\infty$ connected, then the solution converges in $\cSM$ to this equilibrium.
\par
Conversely, if $(u(t),\theta(t),\Gamma(t))$ is a global solution in $\cSM$ which converges to an equilibrium $(0,\theta_*,\Gamma_*)\in\cE$  in $\cSM$ as $t\to\infty$, then (i), (ii) and (iii) are valid.
\end{thm}
\begin{proof}
The assertions follow by similar arguments as in the  proof of Theorem~6.2 in \cite{PSW13}.
\end{proof}

\end{document}